\documentclass[amstex,12pt, amssymb]{article}

\usepackage{mathtext}
\usepackage[cp1251]{inputenc}
\usepackage[T2A]{fontenc}
\usepackage[dvips]{graphicx}
\usepackage{amsmath}
\usepackage{amssymb}
\usepackage{amsxtra}
\usepackage{latexsym}
\usepackage{ifthen}

\newcounter{lemma}[section]

\newcounter{corollary}[section]

\newcounter{remark}[section]

\newcounter{theorem}[section]

\newcounter{proposition}[section]

\numberwithin{equation}{section}

\begin{document}

\markboth{\centerline{E.~SEVOST'YANOV}}
{\centerline{ON BOUNDARY EXTENSION }}

\def\cc{\setcounter{equation}{0}
\setcounter{figure}{0}\setcounter{table}{0}}

\overfullrule=0pt


\author{{E.~SEVOST'YANOV}\\}

\title{
{\bf ON BOUNDARY EXTENSION OF MAPPINGS IN METRIC SPACES IN THE TERMS
OF PRIME ENDS}}

\date{\today}
\maketitle

\begin{abstract} We study the
boundary behavior of the so-called ring $Q$-mappings obtained as a
natural generalization of mappings with bounded distortion. We
establish a series of conditions imposed on a function $Q(x)$ for
the continuous extension of given mappings with respect to prime
ends in domains with regular boundaries in metric spaces.
\end{abstract}

\bigskip
{\bf 2010 Mathematics Subject Classification: Primary 30L10;
Secondary 30C65}

\section{Introduction} Problems of continuous extension of mappings
with finite distortion in terms of prime ends in ${\Bbb R}^n$ were
recently investigated in \cite{GRY} for $n=2,$ and in \cite{KR} for
$n\geqslant 2.$ The latter paper was devoted to the case of
homeomorphisms between spatial domains with regular prime ends.
However, the case of mappings with branching was not considered in
these papers. The present paper solves similar problems in general
metric spaces and not only for homeomorphisms but also for more
general open discrete mappings, cf. \cite{ABBS}, \cite{A} and
\cite{Sev$_1$}.

\medskip
The following definitions are from \cite{A} and \cite{ABBS}. Given
a metric space $(X, d, \mu)$ with a measure $\mu,$ a {\it domain}
in $X$ is an open path-connected set in $X.$ Recall that $X$ is
locally (path) connected if every neighborhood of a point $x\in X$
contains a (path) connected neighborhood. We define the
Mazurkiewicz distance $d_M$ on $X$ by $d_M(x, y)=\inf {\rm diam}\,
E,$ where the infimum is over all connected sets $E\subset X$
containing $x, y\in E.$ Clearly, $d_M$ is a metric on $X.$ Let
$\gamma$ be a curve in $\Omega.$ We define its diameter as
follows:
$${\rm diam}\,\gamma:= \sup d(x, y),$$
where the supremum is taken over all points $x, y\in \gamma.$
When $x,y\in X,$ we have
$$d_M(x, y)\geqslant d(x, y)\,.$$

\medskip
Set $$B(x_0, r):=\{x\in X: d(x, x_0)<r\}\,,\quad S(x_0, r):=\{x\in
X: d(x, x_0)=r\}\,.$$
From now on we assume that the space $X$ is complete and supports a
$p$-Poincare inequality, and that the measure $µ$ is doubling (see
\cite{ABBS}). In this case, a space $X$ is locally connected (see
\cite[Section~2]{ABBS}), and proper (see
\cite[Proposition~3.1]{BB}). If $X$ is also connected then there
exist constants $C>0$ and $q>0$ such that for all  $x\in X,$ $0 <
r\leqslant R$ and $y\in B(x, R),$
\begin{equation}\label{eq2}
\frac{\mu(B(y, r))}{\mu(B(x, R))}\leqslant
C\left(\frac{r}{R}\right)^q\,,
\end{equation}
see \cite[(2.2)]{ABBS}. Let $\Omega\varsubsetneq X$ be a bounded
domain in $X,$ i.e. a bounded nonempty connected open subset of $X$
that is not the whole space $X$ itself. \medskip The completion of
the metric space $(\Omega,  d_M)$ is denoted $\overline{\Omega}^M,$
and $d_M$ extends in the standard way to $\overline{\Omega}^M:$  For
$d_M$-Cauchy sequences $\{x_n\}_{n=1}^{\infty},$
$\{y_n\}_{n=1}^{\infty}\in \Omega$ we define the equivalence
relation
$$\{x_n\}_{n=1}^{\infty}\sim\{y_n\}_{n=1}^{\infty}\quad \text{if}\quad \lim\limits_{n\rightarrow\infty}
d_M(x_n, y_n)=0\,.$$
Note that every Cauchy sequence is trivially equivalent to any of
its subsequences.

The collection of all equivalence classes of $d_M$-Cauchy sequences
can be formally considered to be $\overline{\Omega}^M,$ but we will
identify equivalence classes of $d_M$-Cauchy sequences having a
limit in $\Omega$ with that limit point. By considering equivalence
classes of $d_M$-Cauchy sequences without limits in $\Omega$ we
define the boundary of $\Omega$ with respect to $d_M$ as
$\partial_M\Omega=\overline{\Omega}^M\setminus\Omega.$ Since $X$ is
proper, we know that $\Omega$ is locally compact with respect to
$d_M,$ and it follows that $\Omega$ is an open subset of
$\overline{\Omega}^M.$ We extend the original metric $d_M$ on
$\Omega$ to $\overline{\Omega}^M$ by setting
$$d_M(x^*, y^*)=\lim\limits_{n\rightarrow\infty} d_M(x_n, y_n)\,,$$
if $x^*=\{x_n\}_{n=1}^{\infty}\in\overline{\Omega}^M$ and
$y^*=\{y_n\}_{n=1}^{\infty}\in\overline{\Omega}^M.$ This is well
defined and an extension of $d_M.$

\medskip
We call a bounded connected set $E \varsubsetneq \Omega $ an {\it
acceptable} set if $\overline{E}\cap\partial\Omega\ne\varnothing.$
By discussion in \cite{ABBS}, we know that boundedness and
connectedness of an acceptable set $E$ implies that $\overline{E}$
is compact and connected. Furthermore, $E$ is infinite, as otherwise
we would have $\overline{E}=E\subset \Omega.$ Therefore,
$\overline{E}$ is a continuum. Recall that a {\it continuum} is a
connected compact set containing at least two points.

\medskip
We call a sequence $\left\{E_k\right\}_{k=1}^{\infty}$ of
acceptable sets a chain if it satisfies the following conditions:

1. $E_{k+1}\subset E_k$ for all $k=1, 2,\ldots,$

2. ${\rm dist}\,(\Omega\cap \partial E_{k+1}, \Omega\cap \partial
E_k)>0$ for all $k=1, 2,\ldots,$

3. The impression
$\bigcap\limits_{k=1}^{\infty}\overline{E_k}\subset \partial
\Omega.$

\begin{figure}[h]
\centerline{\includegraphics[scale=0.6]{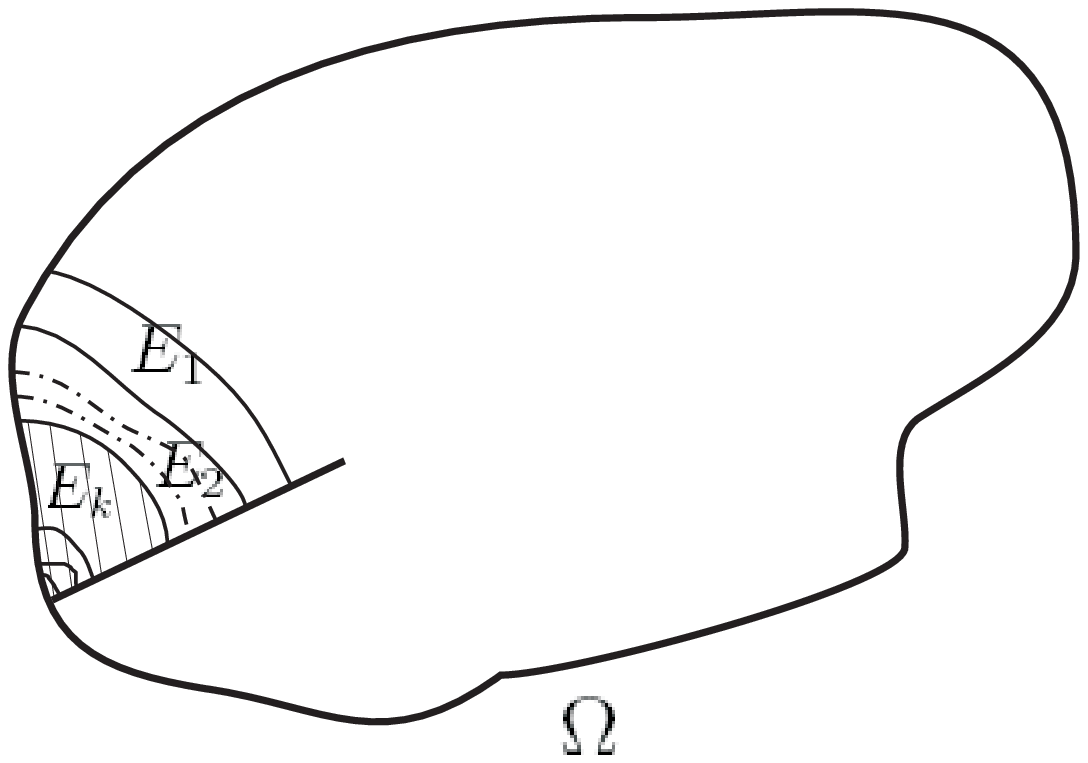}}
\centerline{Picture~1. A prime end in domain $\Omega$}
\end{figure}

\medskip
We say that a chain $\left\{E_k\right\}_{k=1}^{\infty}$ divides the
chain $\left\{F_k\right\}_{k=1}^{\infty}$ if for each $k$ there
exists $l_k$ such that $E_{l_k}\subset F_k.$ (2) Two chains are
equivalent if they divide each other. A collection of all mutually
equivalent chains is called an {\it end} and denoted $[E_k],$ where
$\left\{E_k\right\}_{k=1}^{\infty}$ is any of the chains in the
equivalence class. The impression of $[E_k],$ denoted $I[E_k],$ is
defined as the impression of any representative chain. The
collection of all ends is called the {\it end boundary} and is
denoted $\partial_E\Omega.$ We say that an end $[E_k]$ is a {\it
prime end} if it is not divisible by any other end. The collection
of all prime ends is called the {\it prime end boundary} and is
denoted $E_{\Omega}.$

\medskip
We say that a sequence of points $\{x_n\}_{n=1}^{\infty}$ in
$\Omega$ converges to the end $[E_k],$ and write $x_n\rightarrow
[E_k]$ as $n\rightarrow\infty,$ if for all $k$ there exists $n_k$
such that $x_n\in E_k$ whenever $n\geqslant n_k.$ If $x_n\rightarrow
[E_k]$ as $n\rightarrow\infty,$ and $[E_k]$ divides $[F_k],$ then
$x_n$ also converges to $[F_k].$ Convergence of points and ends
defines a topology on $\Omega\cup\partial_E\Omega$ (see
e.g.~\cite[Proposition~8.4]{ABBS}). In this topology, a collection
$C\subset\Omega\cup\partial_E\Omega$ of points and ends is closed if
whenever (a point or an end) $y\in\Omega\cup\partial_E\Omega$ is a
limit of a sequence in $C,$ then $y\in C.$

\medskip
In what follows, we set $\overline{\Omega}^P:=\Omega\cup
E_{\Omega}.$ We say that $\Omega$ is {\it finitely connected} at a
point $x_0\in\partial \Omega$ if for every $r>0$ there is an open
set $G$ (open in $X$) such that $x_0\in G \subset B(x_0, r)$ and
$G\cap\Omega$ has only finitely many components. If $\Omega$ is
finitely connected at every boundary point, then it is called {\it
finitely connected} at the boundary. The following results have been
proved in~\cite{ABBS}.

\medskip
\begin{proposition}\label{pr3}
{\sl Assume that $\Omega$ is finitely connected at the boundary.
Then all prime ends have singleton impressions, and every $x\in
\partial\Omega$ is the impression of a prime end and is accessible (see~\cite[Theorem~10.8]{ABBS}). }
\end{proposition}

\medskip
\begin{proposition}\label{pr2}
{\sl Assume that $\Omega$ is finitely connected at the boundary.
Then there is a homeomorphism $\Phi:\overline{\Omega}^P\rightarrow
\overline{\Omega}^M$ such that $\Phi|_{\Omega}$ is the identity map.
Moreover, the prime end closure $\overline{\Omega}^P$ is metrizable
with the metric $m_P(x, y):= d_M(\Phi(x), \Phi(y)).$ The topology on
$\overline{\Omega}^P$ given by this metric is equivalent to the
topology given by the sequential convergence discussed above
(see~\cite[Corollary~10.9]{ABBS}).}
\end{proposition}

\medskip
Recall, for a given continuous path $\gamma:[a, b]\rightarrow X$ in
a metric space $(X, d),$ that its length is the supremum of the sums
$$
\sum\limits^{k}_{i=1} d(\gamma(t_i),\gamma(t_{i-1}))
$$
over all partitions $a=t_0\leqslant t_1\leqslant\ldots\leqslant
t_k=b$ of the interval $ [a,b].$ The path $\gamma$ is called {\it
rectifiable} if its length is finite.
\medskip

\medskip
Given a family of paths $\Gamma$ in $X$, a Borel function
$\varrho:X\rightarrow[0,\infty]$ is called {\it admissible} for
$\Gamma$, abbr. $\varrho\in {\rm adm}\,\Gamma$, if
%
$$\int\limits_{\gamma}\varrho\,ds\ \geqslant\ 1$$
for all (locally rectifiable) $\gamma\in\Gamma$.
\medskip
Everywhere further, for any sets $E, F,$ and $G$ in $X$, we denote
by $\Gamma(E, F, G)$ the family of all continuous curves $\gamma:[0,
1]\rightarrow X$ such that $\gamma(0)\in E,$ $\gamma(1)\in F,$ and
$\gamma(t)\in G$ for all $t\in (0, 1).$ Everywhere further $(X, d,
\mu)$ and $\left(X^{\,\prime}, d^{\,\prime}, \mu^{\,\prime}\right)$
are metric spaces with metrics $d$ and $d^{\,\prime}$ and locally
finite Borel measures $\mu$ and $\mu^{\,\prime},$ correspondingly.
We will assume that $\mu$ is a Borel measure such that $0
<\mu(B)<\infty$ for all balls $B$ in $X.$

\medskip
Given $p\geqslant 1,$ the $p$-modulus of the family $\Gamma$ is the
number
%
$$M_p(\Gamma)=\inf\limits_{\rho\in {\rm
adm}\,\Gamma}\int\limits_{X}\varrho^{\,p}(x)\, d\mu(x)\,.$$
%
Should ${\rm adm\,}\Gamma$ be empty, we set $M_p(\Gamma)=\infty.$ A
family of paths $\Gamma_1$ in $X$ is said to be {\it minorized} by a
family of paths $\Gamma_2$ in $X,$ abbr. $\Gamma_1>\Gamma_2,$ if,
for every path $\gamma_1\in\Gamma_1$, there is a path
$\gamma_2\in\Gamma_1$ such that $\gamma_2$ is a restriction of
$\gamma_1.$ In this case,
\begin{equation}\label{eq32*A}
\Gamma_1
> \Gamma_2 \quad \Rightarrow \quad M_p(\Gamma_1)\le M_p(\Gamma_2)
\end{equation} (см. \cite[Theorem~1]{Fu}).

\medskip
Let $G$ and $G^{\,\prime}$ be domains with finite Hausdorff
dimensions $\alpha$ and $\alpha^{\,\prime}\geqslant 1$ in spaces
$(X,d,\mu)$ and $(X^{\,\prime},d^{\,\prime}, \mu^{\,\prime}),$ and
let $Q:G\rightarrow[0,\infty]$ be a measurable function. Given
$x_0\in
\partial G,$ denote $S_i:=S(x_0, r_i),$ $i=1,2,$ where $0<r_1<r_2<\infty.$ We say that a mapping
$f:G\rightarrow G^{\,\prime}$ is a {\it ring $Q$-mapping at a point
$x_0\in \partial G$}, if the inequality
\begin{equation}\label{eq1C}
M_{{\,\alpha}^{\,\prime}}(f(\Gamma(S_1, S_2,
A)))\leqslant\int\limits_{A\cap G}Q(x)\eta^{\alpha}(d(x,
x_0))\,d\mu(x)
\end{equation}
holds for any ring
\begin{equation}\label{eq15}
A=A(x_0, r_1, r_2)=\{x\in X: r_1<d(x, x_0)<r_2\}, \quad 0 < r_1 <
r_2 <\infty\,,
\end{equation}
and any measurable function
$\eta:(r_1, r_2)\rightarrow [0, \infty]$ such that
%
$$\int\limits_{r_1}^{r_2}\eta(r)dr\geqslant 1$$
%
holds. We also consider the definition (\ref{eq1C}) for maps
$f:G\rightarrow X^{\,\prime},$ where $G\subset X$ is a domain of
Hausdorff dimension $\alpha,$ and $X^{\,\prime}$ is a metric space
of Hausdorff dimension $\alpha^{\,\prime}.$

\medskip
\begin{remark}\label{rem1}
Sometimes, some another (similar) definition of ring $Q$-maps is
considered. Let $G$ and $G^{\,\prime}$ be domains with finite
Hausdorff dimensions $\alpha$ and $\alpha^{\,\prime}\geqslant 1$ in
spaces $(X,d,\mu)$ and $(X^{\,\prime},d^{\,\prime},
\mu^{\,\prime}),$ and let $Q:G\rightarrow[0,\infty]$ be a measurable
function. Following to \cite{Sm}, we say that a mapping
$f:G\rightarrow G^{\,\prime}$ is a ring $Q$-mapping at a point
$x_0\in \overline{G},$ if the inequality
\begin{equation}\label{eq1B}
M_{{\,\alpha}^{\,\prime}}(f(\Gamma(C_1, C_0,
G)))\leqslant\int\limits_{A\cap G}Q(x)\eta^{\alpha}(d(x,
x_0))d\mu(x)
\end{equation}
holds for any ring
%
$$A=A(x_0, r_1, r_2)=\{x\in X: r_1<d(x, x_0)<r_2\}, \quad 0 < r_1 <
r_2 <\infty\,,$$
%
and any two continua $C_0\subset \overline{B(x_0, r_1)}\cap G,$
$C_1\subset G\setminus B(x_0, r_2),$ and any measurable function
$\eta:(r_1, r_2)\rightarrow [0, \infty]$ such that
\begin{equation}\label{eq*3!!}
\int\limits_{r_1}^{r_2}\eta(r)dr\geqslant 1
\end{equation}
holds.

\medskip
Observe that (\ref{eq1C}) implies (\ref{eq1B}). In fact, assume that
(\ref{eq1C}) holds.  Let $C_0\subset \overline{B(x_0, r_1)}\cap G,$
$C_1\subset G\setminus B(x_0, r_2)$ be two continua. Assume that
$\gamma\in \Gamma(C_1, C_0, G).$ Given a curve $\gamma:[0,
1]\rightarrow G,$ we set $|\gamma|:=\{x\in G: \exists\,t\in [0, 1]:
\gamma(t)=x\}.$ Note that $|\gamma|$ is not included entirely both
in $B(x_0, r_2)$ and $G\setminus B(x_0, r_2),$ therefore there
exists $y_1\in S(x_0, r_2)$ (see \cite[Theorem 1, $\S\,$46, item
I]{Ku}). Let $\gamma:[0, 1]\rightarrow G$ and let $t_1\in (0, 1)$ be
such that $\gamma(t_1)=y_1.$ There is no loss of generality in
assuming that $|\gamma|_{[0, t_1)}|\subset B(x_0, r_2).$ We put
$\gamma_1:=\gamma|_{[0, t_1)}.$ Observe that $|\gamma_1|\subset
B(x_0, r_2),$ moreover, $\gamma_1$ is not included entirely either
in $\overline{B(x_0, r_1)}$ or in $G\setminus \overline{B(x_0,
r_1)}.$ Consequently, there exists $t_2\in (0, t_1)$ with
$\gamma_1(t_2)\in S(x_0, r_1)$ (see \cite[Theorem 1, $\S\,$46, item
I]{Ku}). There is no loss of generality in assuming that
$|\gamma_1|_{[t_2, t_1]}|\subset G\setminus\overline{B(x_0, r_1)}.$
Put $\gamma_2=\gamma_1|_{[t_2, t_1]}.$ Observe that $\gamma_2$ is a
subcurve of $\gamma.$ By the said above, $\Gamma(C_1, C_0,
G)>\Gamma(S_1, S_2, A)$ and, consequently, $f(\Gamma(C_1, C_0,
G))>f(\Gamma(S_1, S_2, A)).$ Now, by (\ref{eq32*A}),
\begin{equation}\label{eq9}
M_{\alpha^{\,\prime}}(f(\Gamma(C_1, C_0, G)))\leqslant
M_{\alpha^{\,\prime}}(f(\Gamma(S_1, S_2, A)))\,.
\end{equation}
Combining (\ref{eq1C}) with (\ref{eq9}), we obtain (\ref{eq1B}).
\end{remark}

\medskip
We say that the boundary of the domain $G$ is {\it strongly
accessible at a point $x_0\in
\partial G$}, if, for every neighborhood  $U$ of the point $x_0$,
there is a compact set $E\subset G$, a neighborhood $V\subset U$ of
the point $x_0$ and a number $\delta
>0$ such that $$M_{\alpha}(\Gamma(E, F, G))\geqslant \delta$$ for every continuum
$F$ in $G$ intersecting $\partial U$ and $\partial V.$ We say that
the boundary $\partial G$  is {\it strongly accessible}, if the
corresponding property holds at every point of the boundary.

\medskip
Let $X$ and $Y$ be metric spaces. A mapping $f:X\rightarrow Y$ is
discrete if $f^{\,-1}(y)$ is discrete for all $y\in Y$ and $f$ is
open if it takes open sets onto open sets. Given a domain $D\subset
X,$ the {\it cluster set} of $f:D\rightarrow Y$ at $b\in \partial D$
is the set $C(f, b)$ of all points $z\in Y$ for which there exists a
sequence $\{b_k\}_{k=1}^{\infty}$ in $D$ such that $b_k\rightarrow
b$ and $f(b_k)\rightarrow z$ as $k\rightarrow\infty.$ For a
non-empty set $E\subset \partial D$ let $C(f, E)=\cup C(f, b),$
where $b$ ranges over set $E.$ A mapping $f:G\rightarrow Y$ is {\it
closed} in $G\subset X$ if $f(A)$ is closed in $f(G)$ whenever $A$
closed in $G.$ A mapping $f$ is {\it proper} if $f^{\,-1}(K)$ is
compact in $D$ whenever $K$ is a compact set of $f(D).$ A mapping
$f$ is {\it boundary preserving } if $C(f,
\partial D)\subset \partial f(D).$

\medskip
Let $D\subset X,$ $f:D \rightarrow X^{\,\prime}$ be a discrete open
mapping, $\beta: [a,\,b)\rightarrow X^{\,\prime}$ be  a curve, and
$x\in\,f^{-1}\left(\beta(a)\right).$ A curve $\alpha:
[a,\,c)\rightarrow D$ is called a {\it maximal $f$-lifting} of
$\beta$ starting at $x,$ if $(1)\quad \alpha(a)=x\,;$ $(2)\quad
f\circ\alpha=\beta|_{[a,\,c)};$ $(3)$\quad for
$c<c^{\prime}\leqslant b,$ there is no curves $\alpha^{\prime}:
[a,\,c^{\prime})\rightarrow D$ such that
$\alpha=\alpha^{\prime}|_{[a,\,c)}$ and $f\circ
\alpha^{\,\prime}=\beta|_{[a,\,c^{\prime})}.$ In the case
$X=X^{\,\prime}={\Bbb R}^n,$ the assumption on $f$ yields that every
curve $\beta$ with $x\in f^{\,-1}\left(\beta(a)\right)$ has  a
maximal $f$-lif\-ting starting at $x$ (see
\cite[Corollary~II.3.3]{Ri}, \cite[Lemma~3.12]{MRV}). Consider the
condition

\medskip\medskip\medskip
$\textbf{A}:$ {\bf for all $\beta: [a,\,b)\rightarrow X^{\,\prime}$
and $x\in f^{\,-1}\left(\beta(a)\right),$ a mapping $f:D\rightarrow
X^{\,\prime}$ has a maximal $f$-lif\-ting in $D$ starting at $x.$}

\medskip Let $G$ be a domain in a space $(X,d,\mu)$. Similarly to
\cite{IR}, we say that a function $\varphi:G\rightarrow{\Bbb R}$ has
{\it finite mean oscillation at a point $x_{0}\in\overline{G}$},
abbr. $\varphi \in FMO(x_{0})$, if
\begin{equation}\label{eq13.4.111} \overline{\lim\limits_{\varepsilon\rightarrow
0}}\,\, \,\frac{1}{\mu(B(x_{0},\varepsilon))}
\int\limits_{B(x_{0},\varepsilon)}|\varphi(x)-\overline{\varphi}_{\varepsilon}|\,\,d\mu(x)<\infty
\end{equation}
where $$\overline{\varphi}_{\varepsilon}
=\frac{1}{\mu(B(x_{0},\varepsilon))}
\int\limits_{B(x_{0},\varepsilon)}\varphi(x)\,\,d\mu(x)$$ is the
mean value of the function $\varphi(x)$ over the set
$$B(x_{0},\varepsilon)=\{x\in G: d(x,x_0)<\varepsilon\}$$ with
respect to the measure $\mu$. Here the condition (\ref{eq13.4.111})
includes the assumption that $\varphi$ is integrable with respect to
the measure $\mu$ over the set $B(x_0,\varepsilon)$ for some
$\varepsilon>0$.

\medskip
The following result holds.

\medskip
\begin{theorem}\label{th4}
{\sl\, Let $D$ and $D^{\,\prime}$ be domains with finite Hausdorff
dimensions $\alpha$ and $\alpha^{\,\prime}\geqslant 2$ in spaces
$(X,d,\mu)$ and $(X^{\,\prime},d^{\,\prime}, \mu^{\,\prime}),$
respectively. Assume that $X$ is complete and supports an
$\alpha$-Poincare inequality, and that the measure $µ$ is doubling.
Let $D$ be a bounded domain which is finitely connected at the
boundary, and let $Q:X\rightarrow (0, \infty)$ be a locally
integrable function. Suppose that $f:D\rightarrow D^{\,\prime},$
$D^{\,\prime}=f(D),$ is a discrete, closed and open ring $Q$-mapping
in $\partial D,$ for which ${\bf A}$-condition holds. Moreover,
suppose that $\partial D^{\,\prime}$ is strongly accessible and
$\overline{D^{\,\prime}}$ is compact in $X^{\,\prime}.$ Then $f$ has
a continuous extension $f:\overline{D}_P\rightarrow
\overline{D^{\,\prime}},$
$f(\overline{D}_P)=\overline{D^{\,\prime}},$ whenever $Q\in
FMO(\partial D).$ }
\end{theorem}

\medskip
By correspondence $[E_k]\mapsto f([E_k]),$ $[E_k]\in E_D,$
$f([E_k])\in \partial D^{\,\prime},$ we mean the following. If $x_k$
is a sequence with $x_k\rightarrow [E_k],$ $k\rightarrow\infty,$
then we set: $f([E_k]):=\lim\limits_{k\rightarrow\infty}f(x_k).$ The
statement of the Theorem~\ref{th4} includes that this limit exists,
and it does not depend on a sequence $x_k,$ which converges to
$[E_k]$ (see the Picture~2).

\begin{figure}[h]
\centerline{\includegraphics[scale=0.6]{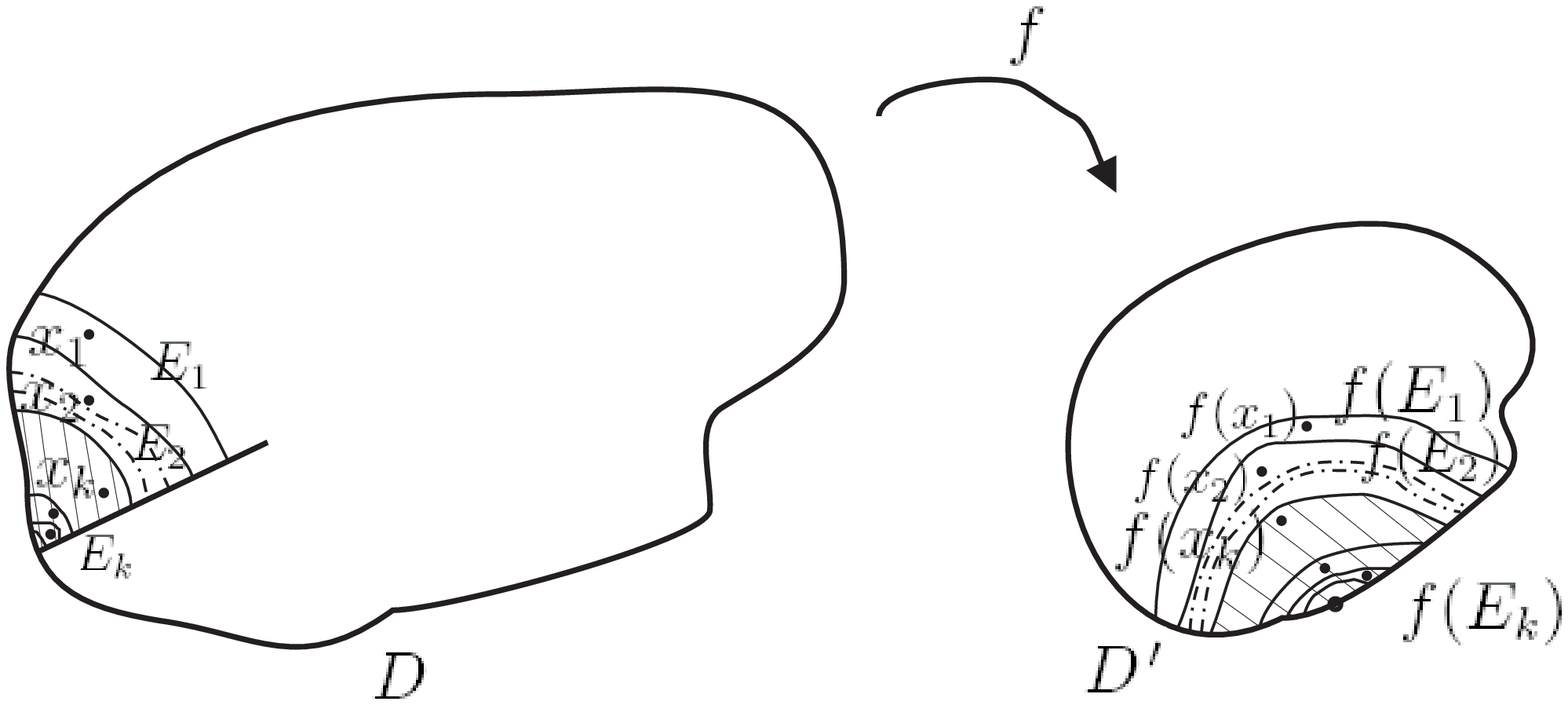}}
\centerline{Picture~2. A correspondence of prime ends and boundary
points under a mapping}
\end{figure}

\section{Main Lemma}

\medskip
The following statement holds (see also \cite[Theorem~3.3]{Vu} for
space ${\Bbb R}^n$).

\medskip
\begin{proposition}\label{pr4}{\sl\, Let $(X, d, \mu)$ be metric space
with Borel measure $\mu,$ and let $D$ be a domain in $X.$ Assume
that the measure $µ$ is doubling and $0 <\mu(B)<\infty$ for all
balls $B$ in $X.$ If $f:D\rightarrow X^{\,\prime}$ is a discrete,
closed and open mapping of $D$ onto a set $D^{\,\prime},$ then $f$
is boundary preserving. Moreover, $f^{\,-1}(K)$ is a compact for
every compact set $K\subset D^{\,\prime}.$ }
\end{proposition}

\medskip
\begin{proof}
Since $f$ is open, $D^{\,\prime}$ is a domain. Assume, to the
contrary, that $f$ is not boundary preserving. Then there exists
$x_0\in
\partial D$ and $y\in D^{\,\prime}$ such that $y\in C(f, x_0).$ Now,
we can find a sequence $x_k\rightarrow x_0$ as $k\rightarrow\infty,$
$x_k\in D,$ $k=1,2,\ldots ,$ such that $f(x_k)\rightarrow y$ as
$k\rightarrow\infty.$

\medskip
Without loss of generality, we can consider that $f(x_k)\ne y$ for
all $k=1,2\ldots .$ In fact, by continuity of $f,$ for every $k\in
{\Bbb N}$ there exists $\delta_k>0$ such that
\begin{equation}\label{eq3}
d^{\,\prime}(f(x), f(x_k))<1/k\quad \forall \,\,x\in B(x_k,
\delta_k)\,.
\end{equation}
We can consider that $B(x_k, \delta_k)\subset D$ and $\delta_k<1/k.$
Letting to the limit as $r\rightarrow 0$ in (\ref{eq2}) at
$y=x=x_0,$ we obtain that $\mu(\{x_k\})=0.$ Fix $i\in {\Bbb N}.$
Since, by assumption on $\mu,$ $\mu(B(x_k, \delta_k/2^{\,i}))>0,$ we
obtain that $B(x_k, \delta_k/2^{\,i})$ contains at least two points.
By increasing of $i,$ $i=1,2,\ldots,$ we obtain a sequence
$x_{ik}\in B(x_k, \delta_k/2^{\,i})$ such that $x_{ik}\rightarrow
x_k$ as $i\rightarrow\infty$ and $x_{ik}\ne x_k$ for every $i\in
{\Bbb N}.$ By discreteness of $f,$ we can consider that
$f(x_{ik})\ne y_0$ for all $i\in {\Bbb N}.$ Fix some such $i_0\in
{\Bbb N}$ and set $z_k:=x_{i_0k}.$ Now, by triangle inequality,
$$d(z_k, x_0)\leqslant d(z_k, x_k)+d(x_k, x_0)\rightarrow 0\,,\quad
k\rightarrow\infty\,,$$ and, simultaneously, by (\ref{eq3})
$$d^{\,\prime}(f(z_k), y)\leqslant d^{\,\prime}(f(z_k), f(x_k))+ d^{\,\prime}(f(x_k), y)<$$
$$<1/k+d^{\,\prime}(f(x_k), y)\rightarrow 0,\quad k\rightarrow\infty\,.$$

\medskip
So, $z_k\in D,$ $z_k\rightarrow x_0$ as $k\rightarrow\infty,$
$f(z_k)\rightarrow y$ as $k\rightarrow\infty,$ and $f(z_k)\ne y$ for
every $k\in {\Bbb N}.$ On the other hand, note that
$\{x_k\}_{k=1}^{\infty}$ is closed in $D,$ but
$\{f(x_k)\}_{k=1}^{\infty}$ is not closed in $f(D),$ because
$y\not\in \{f(x_k)\}_{k=1}^{\infty}.$ Now $f$ is not closed in $D$
that contradicts to conditions of Proposition. The contradiction
obtained above disproves that $f$ is not boundary preserving.

\medskip
It remains to show that $f^{\,-1}(K)$ is a compact for every compact
set $K\subset D^{\,\prime}.$ If this is not true, there exists a
sequence $x_k\in f^{\,-1}(K),$ such that $x_k\rightarrow x_0\in
\partial D.$ As was shown above, $f(x_k)\rightarrow y_0\in \partial D^{\,\prime},$ that contradicts to condition
$x_k\in f^{\,-1}(K).$~$\Box$
\end{proof}

\medskip
The following statement was proved in \cite[Lemma~5.1]{GRY} for
homeomorphisms in ${\Bbb R}^2.$

\medskip
\begin{lemma}\label{lem1}
{\sl Let $D$ and $D^{\,\prime}$ be domains with finite Hausdorff
dimensions $\alpha$ and $\alpha^{\,\prime}\geqslant 2$ in spaces
$(X,d,\mu)$ and $(X^{\,\prime},d^{\,\prime}, \mu^{\,\prime}),$
respectively. Assume that $X$ is complete and supports an
$\alpha$-Poincare inequality, and that the measure $µ$ is doubling.
Let $D$ be a bounded domain which is finitely connected at the
boundary, and let $Q:X\rightarrow (0, \infty)$ be a locally
integrable function. Suppose that $f:D\rightarrow D^{\,\prime},$
$D^{\,\prime}=f(D),$ is a discrete, closed and open ring $Q$-mapping
in $\partial D,$ for which ${\bf A}$-condition holds. Moreover,
suppose that $\partial D^{\,\prime}$ is strongly accessible and
$\overline{D^{\,\prime}}$ is compact in $X^{\,\prime}.$  Assume
that, for every $x_0\in
\partial D,$ there exists a Lebesgue measurable function $\psi:(0,
\infty)\rightarrow (0, \infty)$ such that
\begin{equation} \label{eq5B}
I(\varepsilon,
\varepsilon_0):=\int\limits_{\varepsilon}^{\varepsilon_0}\psi(t)dt <
\infty
\end{equation}
for every $\varepsilon\in (0,\varepsilon_0)$ and $I(\varepsilon,
\varepsilon_0)\rightarrow\infty$ as $\varepsilon\rightarrow 0,$ and

\begin{equation} \label{eq4*}
\int\limits_{\varepsilon<d(x,
x_0)<\varepsilon_0}Q(x)\cdot\psi^{\,\alpha}(d(x, x_0)) \
d\mu(x)\,=\,o\left(I^{\,\alpha}(\varepsilon, \varepsilon_0)\right)
\end{equation}
as $\varepsilon\rightarrow 0.$ Then $f$ has a continuous extension
$f:\overline{D}_P\rightarrow \overline{D^{\,\prime}},$
$f(\overline{D}_P)=\overline{D^{\,\prime}}.$ }
\end{lemma}

\begin{proof} By Proposition \ref{pr2}, $\overline{D}_P$ is metrizable. Now,
by metrizability of $\overline{D}_P,$ it is sufficient to prove that
$$L=C(f, P):=\left\{y\in X^{\,\prime}: y=\lim\limits_{k\rightarrow\infty}f(x_k),x_k\rightarrow
P,x_k\in D\right\}$$
consists of single point $y_0\in\partial D^{\,\prime}.$ 
Since $\overline{D^{\,\prime}}$ is a compact, $L\ne\varnothing.$ By
Proposition \ref{pr4}, $L\subset\partial D^{\,\prime}.$

\medskip
Assume, to the contrary, that $f$ cannot be extended to $P$
continuously. Now, we can find at least two points $y_0$ and $z_0\in
L$. Set $U=B(y_0,r_0)$, where $0<r_0<d(y_0, z_0)$. Now we can find a
sequences $y_k$ and $z_k$ in $f(E_k),$ $k=1,2,\ldots,$ $P=[E_k],$
such that $d(y_0, y_k)<r_0$ and $d(y_0, z_k)>r_0$ and, besides that,
$y_k\rightarrow y_0$ and $z_k\rightarrow z_0$ as
$k\rightarrow\infty$. By Remark 4.5 in \cite{ABBS} we can consider
that the sets $E_k$ are open. Moreover, by Remark 2.6 in \cite{ABBS}
the set $E_k$ is path connected for every $k\in {\Bbb N}.$

\medskip Denote $x_0:=I([E_k])$ (see Proposition \ref{pr3}).
Now we show that, for every $r>0$ there exists $k\in {\Bbb N}$ such
that
\begin{equation}\label{eq1}
E_k\subset B(x_0, r)\cap D\,.
\end{equation}
Assume, to the contrary, that there exists $r>0$ with the following
condition: for every $k\in {\Bbb N}$ there exists $x_k\in
E_k\setminus B(x_0, r).$ Since $\mu$ is doubling, $X$ is complete if
and only if it is proper (i.e. every closed bounded set is compact),
see \cite[Proposition~3.1]{BB}. Since $D$ is bounded, $\overline{D}$
is compact. Now, we can find a subsequence $x_{k_l}\in D$ with
$x_{k_l}\rightarrow \overline{x_0}$ as $l\rightarrow\infty$ for some
$\overline{x_0}\in\overline{D}.$ Given $i\in{\Bbb N},$ there exists
$l_0\in {\Bbb N}$ such that $k_l>i$ for every $l\geqslant l_0.$
Consequently, $x_{k_l}\in E_{k_l}\subset E_i$ for every $l\geqslant
l_0$ and thus, $\overline{x_0}\in \overline{E_i}.$ Since $i$ is
arbitrary, we obtain that $\overline{x_0}\in
\bigcap\limits_{i=1}^{\infty} \overline{E_i}=\{x_0\}.$ So,
$x_0=\overline{x_0}.$ It remains to show that $x_k\rightarrow x_0$
as $k\rightarrow\infty.$ Assume the contrary, then there exists a
subsequence $x_{m_l}\in D$ with $x_{m_l}\rightarrow \zeta_0$ as
$l\rightarrow\infty.$ Arguing as above, we obtain that
$\zeta_0=x_0,$ that disproves the contradiction mentioned above. Now
$x_k\rightarrow x_0$ as $k\rightarrow\infty$ and thus, $x_k\in
B(x_0, r).$ The inclusion (\ref{eq1}) have been proved.

\medskip
Since $y_k, z_k\in f(E_k),$ one can find at least two sequences
$x_k, x^{\,\prime}_k\in E_k$ such that $f(x_k)=y_k$ and
$f(x_k^{\,\prime})=z_k.$ By (\ref{eq1}) $x_k\rightarrow x_0$ and
$x^{\,\prime}_k\rightarrow x_0$ as $k\rightarrow\infty.$ According
to the definition of a strongly accessible boundary at a point
$y_0\in \partial D^{\,\prime},$ for any neighborhood $U$ of this
point one can find a compact set $C_0 \subset \partial
D^{\,\prime},$ a neighborhood $V$ of the point $y_0$ and a number
$\delta>0$ such that
\begin{equation}\label{eq1A}
M_{\alpha^{\,\prime}}(\Gamma(C_0^{\,\prime}, F,
D^{\,\prime}))\geqslant \delta
>0
\end{equation}
for an arbitrary continuum $F$ that intersects $\partial U$ and
$\partial V.$ By Proposition \ref{pr4},
$C:=f^{\,-1}(C_0^{\,\prime})$ is compact subset of $D.$
Consequently, $\delta_0={\rm dist}(x_0, C)>0.$ Then, without loss of
generality, we can assume that $C_0\cap\overline{B(x_0,
\varepsilon_0)}=\varnothing.$ Since $E_k$ is connected, the points
$x_k$ and $x^{\,\prime}_k$ can be connected by a curve $\gamma_k$
lying in $E_k.$ Since $f(x_k)=y_k\in V$ and
$f(x^{\,\prime}_k)=z_k\in D^{\,\prime}\setminus U$ for sufficiently
large $k\in {\Bbb N},$ one can find a number $k_0\in {\Bbb N}$ such
that, by virtue of (\ref{eq1A}),
\begin{equation}\label{eq2A}
M_{\alpha^{\,\prime}}(\Gamma(C_0^{\,\prime}, f(\gamma_k),
D^{\,\prime}))\geqslant \delta
>0
\end{equation}
for all $k\geqslant k_0.$ Let $\Gamma_k$ denote the family of all
semiopen curves $\beta_k:[a, b)\rightarrow D^{\,\prime}$ such that
$\beta_(a)\in f(\gamma_k),$ $\beta_k(t)\in D^{\,\prime}$ for all
$t\in [a, b),$ and
$$\lim\limits_{t\rightarrow b-0}\beta_k(t):=B_i\in C_0^{\,\prime}\,.$$
It is obvious that
\begin{equation}\label{eq4}
M_{\alpha^{\,\prime}}(\Gamma_k)=M_{\alpha^{\,\prime}}
\left(\Gamma\left(C_0^{\,\prime}, f(\gamma_k),
D^{\,\prime}\right)\right)\,.
\end{equation}
For each fixed $k\in {\Bbb N},$ $k\geqslant k_0,$ we consider the
family $\Gamma_k^{\,\prime}$ of maximal liftings $\alpha_k(t):[a,
c)\rightarrow D$ of the family $\Gamma_k$ with origin in the set
$\gamma_k.$ This family exists and is well defined by virtue of
${\bf A}$-condition. First, note that no curve  $\alpha_k(t)\in
\Gamma_k^{\,\prime},$ $\alpha_k:[a, c)\rightarrow D,$ can not tend
to the boundary of the domain $D$ as $t\rightarrow c-0$ by virtue of
the condition $C(f,
\partial D)\subset \partial D^{\,\prime}.$ Then  $C(\alpha_k(t), c)\subset
D.$ Now assume that the curve $\alpha_k(t)$  does not have a limit
as $t\rightarrow c-0.$

Consider
$$G=\left\{x\in X:\, x=\lim\limits_{k\rightarrow\,\infty}
\alpha(t_k)
 \right\}\,,\quad t_k\,\in\,[a,\,c)\,,\quad
 \lim\limits_{k\rightarrow\infty}t_k=c\,.$$
Letting to subsequences, if it is need, we can restrict us by
monotone sequences $t_k.$ For $x\in G,$ by continuity of $f,$
$f\left(\alpha(t_k)\right)\rightarrow\,f(x)$ as
$k\rightarrow\infty,$ where $t_k\in[a,\,c),\,t_k\rightarrow c$ as
$k\rightarrow \infty.$ However,
$f\left(\alpha(t_k)\right)=\beta(t_k)\rightarrow\beta(c)$ as
$k\rightarrow\infty.$ Thus, $f$ is a constant on $G.$ From other
hand, $\overline{\alpha}$ is a compact set, because
$\overline{\alpha}$ is a closed subset of the compact space
$\overline{D}$ (see \cite[Theorem~2.II.4, $\S\,41$]{Ku}). Now, by
Cantor condition on the compact $\overline{\alpha},$ by monotonicity
of $\alpha\left(\left[t_k,\,c\right)\right),$
$$G\,=\,\bigcap\limits_{k\,=\,1}^{\infty}\,\overline{\alpha\left(\left[t_k,\,c\right)\right)}
\ne\varnothing\,,
$$
%
see \cite[1.II.4, $\S\,41$]{Ku}. Now, by \cite[Theorem~5.II.5,
$\S\,47$]{Ku}, $\overline{\alpha}$ is connected. By discreteness of
$f,$ $G$ is a single-point set, and $\alpha\colon
[a,\,c)\rightarrow\,D$ extends to a closed curve $\alpha\colon
[a,\,c]\rightarrow D,$ and $f\left(\alpha(c)\right)=\beta(c).$

\medskip
Therefore, there exists $\lim\limits_{t\rightarrow
c-0}\alpha_k(t)=A_k\in D.$ Observe that, in this case, by the
definition of maximal lifting, we have $c=b.$ Then, on the one hand,
$\lim\limits_{t\rightarrow b-0}\alpha_k(t):=A_k,$ and, on the other
hand, by virtue of the continuity of the mapping $f$ in $D,$
$$f(A_k)=\lim\limits_{t\rightarrow b-0}f(\alpha_k(t))=\lim\limits_{t\rightarrow b-0}
\beta_k(t)=B_k\in C_0^{\,\prime}\,.$$
According to the definition of $C_0,$ this implies that $A_k$
belongs to $C_0.$ We imbed the compact set $C_0$ into a certain
continuum $C_1$ lying completely in the domain $D$ (see Lemma 1 in
\cite{Sm}). Taking a smaller value of $\varepsilon_0>0,$  we can
again assume that $C_1\cap\overline{B(x_0,
\varepsilon_0)}=\varnothing.$ Now we have that
$\Gamma_k^{\,\prime}\subset\Gamma(\gamma_k, C_1, D).$ Passing to a
subsequence, if necessary, we can consider that $x_k$ and
$x^{\,\prime}_k\in B(x_0, 2^{\,-k}).$ Observe that the function
$$\eta(t)=\left\{
\begin{array}{rr}
\psi(t)/I(2^{\,-k}, \varepsilon_0), &   t\in (2^{\,-k},
\varepsilon_0),\\
0,  &  t\in {\Bbb R}\setminus (2^{\,-k}, \varepsilon_0)\,,
\end{array}
\right. $$
where
$I(\varepsilon):=\int\limits_{\varepsilon}^{\varepsilon_0}\psi(t)dt,$
satisfies a normalization condition of the form (\ref{eq5B}).
Therefore, by Remark \ref{rem1} and conditions (\ref{eq5B}) and
(\ref{eq4*}), we get
\begin{equation}\label{eq11*}
M_{\alpha^{\,\prime}}\left(f\left(\Gamma_k^{\,\prime}\right)\right)\leqslant
M_{\alpha^{\,\prime}}(f(\Gamma(\gamma_k, C_1, D)))\leqslant
\Delta(k)\,,
\end{equation}
where $\Delta(k)\rightarrow 0$ as $k\rightarrow \infty.$ However,
$\Gamma_k=f(\Gamma_k^{\,\prime}).$ Therefore, using (\ref{eq11*}),
we conclude that
\begin{equation}\label{eq3A}
M_{\alpha^{\,\prime}}(\Gamma_k)=
M_{\alpha^{\,\prime}}\left(f(\Gamma_k^{\,\prime})\right)\leqslant
\Delta(k)\rightarrow 0\qquad \text{as}\qquad k\rightarrow\infty\,.
\end{equation}
Relation (\ref{eq3A}), together with equality (\ref{eq4}),
contradicts inequality (\ref{eq2A}), which proves the possibility of
continuous extension $f:\overline{D}_P\rightarrow
\overline{D^{\,\prime}}.$

\medskip
It remains to show that $f(\overline{D}_P)=\overline{D^{\,\prime}}.$
It is clear, that $f(\overline{D}_P)\subset\overline{D^{\,\prime}}.$
Now we show the inverse inclusion. Let $\zeta_0\in
\overline{D^{\,\prime}}.$ If $\zeta_0\in D^{\,\prime},$ then there
exists $\xi_0\in D$ with $f(\xi_0)=\zeta_0$ and, consequently,
$\zeta_0\in f(D).$ Assume that $\zeta_0\in
\partial D^{\,\prime}.$ Now there exists $\zeta_m\in D^{\,\prime},$
$\zeta_m=f(\xi_m),$ $\xi_m\in D,$ such that $\zeta_m\rightarrow
\zeta_0$ as $m\rightarrow\infty.$ By \cite[Theorem~10.10]{ABBS},
$\overline{D}_P$ is a compact metric space. Now, we can consider
that $\xi_m\rightarrow P_0$ as $m\rightarrow\infty,$ where $P_0$ is
some prime end in $\overline{D}_P.$ Now $\zeta_0\in
f(\overline{D}_P).$ The inclusion $\overline{D^{\,\prime}}\subset
f(\overline{D}_P)$ has been proved. Consequently,
$f(\overline{D}_P)=\overline{D^{\,\prime}}.$ Lemma is proved.
~$\Box$
\end{proof}

\section{Proof of the main result}

We will say that a space  $(X,d,\mu)$ is {\it upper $\alpha$-regular
at a point} $x_0\in X$ if there is a constant $C> 0$ such that
$$
\mu(B(x_0,r))\leqslant Cr^{\alpha}$$
for the balls $B(x_0,r)$ centered at $x_0\in X$ with all radii
$r<r_0$ for some $r_0>0.$ We will also say that a space  $(X,d,\mu)$
is {\it upper $\alpha$-regular} if the above condition holds at
every point $x_0\in X.$ The following statement can be found in
\cite[Lemma~4.1]{RS}.

\medskip
\begin{proposition}\label{pr3A}
{\sl Let $G$ be a domain Ahlfors $\alpha$-regular metric space $(X,
d, \mu)$ at $\alpha\geqslant 2.$ Assume that $x_0\in \overline{G}$
and $Q:G\rightarrow [0, \infty]$ belongs to $FMO(x_0).$ If
\begin{equation}\label{eq7}
\mu(G\cap B(x_0, 2r))\leqslant
\gamma\cdot\log^{\alpha-2}\frac{1}{r}\cdot \mu(G\cap B(x_0, r))
\end{equation}
for some $r_0>0$ and every $r\in (0, r_0),$ then $Q$
satisfies~\eqref{eq4*} at $x_0$ for some function $F(\varepsilon,
\varepsilon_0)$ such that $G(\varepsilon):=F(\varepsilon,
\varepsilon_0)/I^{\alpha}(\varepsilon, \varepsilon_0)$ obeying
\/{\em:} $G(\varepsilon)\rightarrow 0$ as $\varepsilon\rightarrow
0,$ and $\psi(t):=\frac{1}{t\log\frac{1}{t}}.$}
\end{proposition}

\medskip
{\it Proof of the Theorem~\ref{th4}} follows from Lemma~\ref{lem1}
and Proposition~\ref{pr3A}. Indeed, $X$ is upper regular by
(\ref{eq2}), and (\ref{eq7}) holds because the measure $µ$ is
doubling by assumptions. So, the desired statement follows from
the Lemma~\ref{lem1}.~$\Box$

\section{Homeomorphic extension to the boundary}

Now we prove results about homeomorphic extension of mappings to the
boundary in terms of prime ends.

\medskip
Let us give the following definition (see
\cite[section~13.3]{MRSY}). Let $(X,d,\mu)$ be metric space with
finite Hausdorff dimension $\alpha\geqslant 1.$ We say that the
boundary of $D$ is {\it weakly flat} at a point $x_0\in
\partial D$ if, for every number $P > 0$ and every neighborhood $U$
of the point $x_0,$ there is a neighborhood $V\subset U$ such that
$M_{\alpha}(\Gamma(E, F, D))\geqslant  P$ for all continua $E$ and
$F$ in $D$ intersecting $\partial U$ and $\partial V.$ We say that
the boundary $\partial D$ is weakly flat if the corresponding
property holds at every point of the boundary.
Given $P\in E_D$ and $f:D\rightarrow X^{\,\prime},$ set
$$L=C(f, P):=\left\{y\in X^{\,\prime}:
y=\lim\limits_{k\rightarrow\infty}f(x_k),x_k\rightarrow P,x_k\in
D\right\}\,.$$
Analog of the following lemma was proved in \cite[Lemma~13.4]{MRSY}
(see also \cite[Lemma~4]{KR} and \cite[Lemma~5]{Sm}).

\medskip
\begin{lemma}\label{lem2}
{\sl Let $D$ and $D^{\,\prime}$ be domains with finite Hausdorff
dimensions $\alpha$ and $\alpha^{\,\prime}\geqslant 2$ in spaces
$(X,d,\mu)$ and $(X^{\,\prime},d^{\,\prime}, \mu^{\,\prime}),$
respectively. Assume that $X$ is complete and supports an
$\alpha$-Poincare inequality, and that the measure $µ$ is
doubling. Let $D$ be a bounded domain which is finitely connected
at the boundary, and let $Q:X\rightarrow (0, \infty)$ be
integrable function in $D,$ $Q(x)\equiv 0$ for $x\in X\setminus
D.$ Suppose that $f:D\rightarrow D^{\,\prime},$
$D^{\,\prime}=f(D),$ is a ring $Q$-homeomorphism in $\partial D,$
moreover, suppose that $\partial D^{\,\prime}$ is weakly flat and
$\overline{D^{\,\prime}}$ is compact in $X^{\,\prime}.$ If $P_1$
and $P_2$ are different prime ends in $E_D,$ then $C(f, P_1)\cap
C(f, P_2)=\varnothing.$}
\end{lemma}

\begin{proof}
Assume that $C_1\cap C_2\ne\varnothing$, where $C_i=C(f, P_i),$
$i=1,2$. Now, there exists $y_0\in C_1\cap C_2.$

{\bf I.} Let $P_1=[E_k],$ $k=1,2,\ldots,$ and $P_2=[G_l],$
$l=1,2,\ldots, .$ By Remark 4.5 in \cite{ABBS} we can consider that
the sets $E_k$ and $G_l$ are open. By Remark 2.6 in \cite{ABBS} the
sets $E_k$ and $G_l$ is path connected for every $k, l\in {\Bbb N}.$

Let us to show that there exists $k_0\in {\Bbb N}$ such that
\begin{equation}\label{eq12}
E_k\cap G_k=\varnothing\quad \forall\,\,k\geqslant k_0\,.
\end{equation}
Suppose the contrary, i.e., suppose that for every $l=1,2,\ldots$
there exists an increasing sequence $k_l,$ $l=1,2,\ldots,$ such
that $x_{k_l}\in E_{k_l}\cap G_{k_l},$ $l=1,2,\ldots .$ Now
$x_{k_l}\rightarrow P_1$ and $x_{k_l}\rightarrow P_2,$
$l\rightarrow\infty.$ Let $m_P$ be the metric on $\overline{D}_P$
defined in Proposition \ref{pr2}. By triangle inequality,
$$m_P(P_1, P_2)\leqslant m_P(P_1, x_{k_l})+m_P(x_{k_l}, P_2)
\rightarrow 0,\qquad l\rightarrow\infty\,,$$
that contradicts to Proposition \ref{pr2}. Thus, (\ref{eq12}) holds,
as required.

\medskip
{\bf II.} Denote $x_0:=I([E_k])$ (see Proposition \ref{pr3}).
Arguing as in the proof of Lemma \ref{lem1}, we can show that, for
every $r>0$ there exists $N\in {\Bbb N}$ such that
\begin{equation}\label{eq10} E_k\subset
B(x_0, r)\cap D\quad \forall\,\, k\geqslant N\,.
\end{equation}
Since $D$ is connected and $E_{k_0+1}\ne D,$ we obtain that
$\partial E_{k_0+1}\cap D\ne\varnothing$ (see \cite[Ch.~5, $\S\,$46,
item I]{Ku}). Set $r_0:=d(x_0,
\partial E_{k_0+1}\cap D).$ Since $\overline{E_{k_0}}$ is compact, $r_0>0.$
By (\ref{eq10}), there exists $m_0\in
{\Bbb N},$ $m_0>k_0+1,$ such that
\begin{equation}\label{eq10A} E_k\subset
B(x_0, r_0/2)\cap D\quad \forall\,\, k\geqslant m_0\,.
\end{equation}
{\bf III.} Set $D_0:=E_{m_0+1},$ $D_*:=G_{m_0+1}.$ Let us to show
that
\begin{equation}\label{eq11}
\Gamma(D_0, D_*, D)>\Gamma(S(x_0, r_0/2), S(x_0, r_0), A(x_0, r_0/2,
r_0))\,,
\end{equation}
where $A(x_0, r_1, r_2)$ is defined in (\ref{eq15}).
Assume that $\gamma\in \Gamma(D_0, D_*, D),$ $\gamma:[0,
1]\rightarrow D.$ Set
$$|\gamma|:=\{x\in D: \exists\,t\in[0, 1]:
\gamma(t)=x\}\,.$$
By (\ref{eq12}), $|\gamma|\cap E_{k_0+1}\ne\varnothing\ne
|\gamma|\cap (D\setminus E_{k_0+1}).$ Thus,
\begin{equation}\label{eq13}
|\gamma|\cap \partial E_{k_0+1}\ne\varnothing
\end{equation} (see \cite[Theorem 1, $\S\,$46, item I]{Ku}).
Moreover, observe that
\begin{equation}\label{eq14}
\gamma(1)\not\in \partial E_{k_0+1}\,.
\end{equation}
Suppose the contrary, i.e., that $\gamma(1)\in \partial
E_{k_0+1}.$ By definition of prime end, $\partial E_{k_0+1}\cap
D\subset \overline{E_{k_0}}.$ Since ${\rm dist}\,(D\cap \partial
E_{k+1}, D\cap \partial E_k)>0$ for all $k=1, 2,\ldots,$ we obtain
that $\partial E_{k_0+1}\cap D\subset E_{k_0}.$ Now, we have that
$\gamma(1)\in E_{k_0}$ and, simultaneously, $\gamma(1)\in
G_{m_0+1}\subset G_{k_0}.$ The last relations contradict with
(\ref{eq12}). Thus, (\ref{eq14}) holds, as required.

\medskip
By (\ref{eq10A}), we obtain that $|\gamma|\cap B(x_0,
r_0/2)\ne\varnothing.$ We prove that $|\gamma|\cap (D\setminus
B(x_0, r_0/2))\ne\varnothing.$ In fact, if it is not true, then
$\gamma(t)\in B(x_0, r_0/2)$ for every $t\in [0, 1].$ However, by
(\ref{eq13}) we obtain that $(\partial E_{k_0+1}\cap D)\cap B(x_0,
r_0/2)\ne \varnothing,$ that contradicts to the definition of
$r_0.$ Thus, $|\gamma|\cap (D\setminus B(x_0,
r_0/2))\ne\varnothing,$ as required. Now, by \cite[Theorem 1,
$\S\,$46, item I]{Ku}, there exists $t_1\in (0, 1]$ with
$\gamma(t_1)\in S(x_0, r_0/2).$ We can consider that
$t_1=\max\{t\in [0, 1]: \gamma(t)\in S(x_0, r_0/2)\}.$ We prove
that $t_1\ne 1.$ Suppose the contrary, i.e., suppose that $t_1=1.$
Now, we obtain that $\gamma(t)\in B(x_0, r_0/2)$ for every $t\in
[0, 1).$ From other hand, by (\ref{eq13}) and (\ref{eq14}), we
obtain that $\partial E_{k_0+1}\cap B(x_0, r_0/2)\ne \varnothing,$
which contradicts to the definition of $r_0.$ Thus, $t_1\ne 1,$ as
required. Set $\gamma_1:=\gamma|_{[t_1, 1]}.$

\medskip
By the definition, $|\gamma_1|\cap B(x_0, r_0)\ne\varnothing.$ We
prove that $|\gamma_1|\cap (D\setminus B(x_0,
r_0))\ne\varnothing.$ In fact, assume the contrary, i.e., assume
that $\gamma_1(t)\in B(x_0, r_0)$ for every $t\in [t_1, 1].$ Since
$\gamma(t)\in B(x_0, r_0/2)$ for $t<t_1,$ by (\ref{eq13}) we
obtain that $|\gamma_1|\cap
\partial E_{k_0+1}\ne\varnothing.$ Consequently, $B(x_0, r_0)\cap (\partial E_{k_0+1}\cap D)\ne\varnothing,$
that contradicts to the definition of $r_0.$ Thus, $|\gamma_1|\cap
(D\setminus B(x_0, r_0))\ne\varnothing,$ as required. Now, by
\cite[Theorem 1, $\S\,$46, item I]{Ku}, there exists $t_2\in (t_1,
1]$ with $\gamma(t_2)\in S(x_0, r).$ We can consider that
$t_2=\min\{t\in [t_1, 1]: \gamma(t)\in S(x_0, r_0)\}.$ We put
$\gamma_2:=\gamma|_{[t_1, t_2]}.$ Observe that $\gamma>\gamma_2$ and
$\gamma_2\in\Gamma(S(x_0, r_0/2), S(x_0, r_0), A(x_0, r_0/2, r_0)).$
Thus, (\ref{eq11}) has been proved.

\medskip
{\bf IV.} Consider the function
$$\eta(t)=\left\{
\begin{array}{rr}
2/r_0, &   t\in (r_0/2, r_0),\\
0,  &  t\in {\Bbb R}\setminus (r_0/2, r_0)\,.
\end{array}
\right. $$
Note that $\eta$ satisfies (\ref{eq*3!!}) with $r_1:=r_0/2$ and
$r_2:=r_0.$ Set $S_1:=S(x_0, r_0/2),$ $S_2:=S(x_0, r_0),$ $A:=A(x_0,
r_0/2, r_0).$ Thus, by (\ref{eq32*A}), (\ref{eq1B}) and
(\ref{eq11}), we obtain that
\begin{equation}\label{eq6}
M_{{\,\alpha}^{\,\prime}}(f(\Gamma(D_0, D_*, D)))\leqslant
M_{{\,\alpha}^{\,\prime}}(f(\Gamma(S_1, S_2,
A)))\leqslant\left(\frac{2}{r_0}\right)^{\alpha}\cdot\Vert
Q\Vert_{L^1(D)}<\infty\,.
\end{equation}
Set $M_0:=\left(\frac{2}{r_0}\right)^{\alpha}\cdot\Vert
Q\Vert_{L^1(D)},$ $0<M_0<\infty.$ Now, by (\ref{eq6}) we obtain that
\begin{equation}\label{eq6A}
M_{{\alpha}^{\,\prime}}(f(\Gamma(D_0, D_*, D)))\leqslant M_0\,.
\end{equation}
{\bf V.} Let us to show that there exists $l_0>0$ such that
\begin{equation}\label{eq16}S(y_0,l_0)\cap f(D_0)\ne\varnothing, \quad S(y_0,l_0)\cap
f(D_*)\ne\varnothing\,.
\end{equation}
In fact, since $y_0\in C_1\cap C_2,$ we obtain that $y_0\in
\overline{f(D_0)}.$ Now, given $r_1>0,$ there exists $x_1\in
B(y_0, r_1)\cap f(D_0).$ Similarly, $y_0\in \overline{f(D_*)},$
and there exists $x_2\in B(y_0, r_1)\cap f(D_*).$ Set
$l_0:=\min\{d^{\,\prime}(y_0, x_1), d^{\,\prime}(y_0, x_2)\}.$ We
have that $f(D_0)\cap B(y_0, l_0)\ne\varnothing\ne f(D_0)\setminus
B(y_0, l_0)$ and $f(D_*)\cap B(y_0, l_0)\ne\varnothing\ne
f(D_*)\setminus B(y_0, l_0).$ By \cite[Theorem 1, $\S\,$46, item
I]{Ku} we obtain (\ref{eq16}), as required.

\medskip
Since $\partial D^{\,\prime}$ is weakly flat, there exists
$r_*\in(0, l_0)$ such that
\begin{equation}\label{eq17}
M_{\alpha^{\,\prime}}(\Gamma(E, F, D^{\,\prime}))> M_0
\end{equation}
for each continua $E$ and $F$ in $D^{\,\prime}$ such that $E\cap
S(y_0, l_0)\ne\varnothing\ne E\cap S(y_0,r_*)$ and $F\cap S(y_0,
l_0)\ne\varnothing\ne F\cap S(y_0,r_*).$ By (\ref{eq16}) there
exist curves $c_1$ and $c_2,$ which join $S(y_0, l_0)$ and
$S(y_0,r_*)$ in domains $f(D_0)$ and $f(D_*),$ correspondingly.
Put $E:=c_1$ and $F:=c_2.$ Observe that $\Gamma(c_1, c_2,
D^{\,\prime})\subset f(\Gamma(D_0, D_*, D)).$ Now, by (\ref{eq17})
we obtain that
\begin{equation*}\label{e:9.4a}
M_0< M_{\alpha^{\,\prime}}(\Gamma(c_1, c_2, D^{\,\prime}))\leqslant
M_{\alpha^{\,\prime}}(f(\Gamma(D_0, D_*, D)))\,,\end{equation*} that
contradicts (\ref{eq6A}). Thus, $C(f, P_1)\cap C(f,
P_2)=\varnothing,$ as required.~$\Box$
\end{proof}

\medskip
There are two important statements which follow from Lemma
\ref{lem2}.

\medskip
\begin{theorem}\label{th1}
{\sl Under conditions of Lemma \ref{lem2}, $f^{\,-1}$ has a
continuous extension
$\overline{f^{\,-1}}:\overline{D^{\,\prime}}\rightarrow\overline{D}_P$
such that
$\overline{f^{\,-1}}(\overline{D^{\,\prime}})=\overline{D}_P.$}
\end{theorem}

\medskip
\begin{proof}
Let us to show that, given $\zeta_0\in \partial D^{\,\prime},$ a set
$C(f^{\,-1}, \zeta_0)$ is a single point $\xi_0\in E_D,$ where $E_D$
denotes prime ends space of $D.$ In fact, assume that
$x_k\stackrel{d^{\,\prime}}{\rightarrow} \zeta_0$ as
$k\rightarrow\infty$ and $y_k
\stackrel{d^{\,\prime}}{\rightarrow}\zeta_0$ as
$k\rightarrow\infty.$ By \cite[Theorem~10.10]{ABBS},
$\overline{D}_P$ is a compact metric space. Thus, we can consider
that $f^{\,-1}(x_k)\rightarrow P_1\in E_D$ and
$f^{\,-1}(y_k)\rightarrow P_2\in E_D$ as $k\rightarrow\infty.$ If
$P_1\ne P_2,$ then $\zeta_0\in C(f, P_1)\cap C(f, P_2)$ that
contradicts to the lemma \ref{lem2}.

Thus, we have the extension $\overline{f^{\,-1}}$ of $f^{\,-1}$ on
$\overline{D^{\,\prime}}$ such that $C(f^{\,-1},
\partial D^{\,\prime})\subset \overline{D}_P\setminus D.$ Let us
to show that $C(\overline{f^{\,-1}}, \partial
D^{\,\prime})=\overline{D}_P\setminus D.$ Given $P_0\in E_D,$ we can
find $x_m\rightarrow P_0$ as $m\rightarrow \infty.$ Since $\mu$ is
doubling, $X$ is complete if and only if it is proper (i.e. every
closed bounded set is compact), see \cite[Proposition~3.1]{BB}.
Since $D$ is bounded, $\overline{D}$ is compact. By assumptions of
the theorem, $\overline{D^{\,\prime}}$ is compact, as well. Thus, we
may assume that $x_m\rightarrow x_0\in
\partial D$ and $f(x_m)\stackrel{d^{\,\prime}}{\rightarrow} \zeta_0\in
\partial D^{\,\prime}$ and $m\rightarrow\infty.$ Thus, $P_0\in C(f^{\,-1},
\zeta_0),$ as required.

Finally, let us to show that
$\overline{f^{\,-1}}:\overline{D^{\,\prime}}\rightarrow
\overline{D}_P$ is continuous in $\overline{D^{\,\prime}}.$ In fact,
assume that $\zeta_m\rightarrow \zeta_0$ as $m\rightarrow\infty,$
$\zeta_m, \zeta_0\in \overline{D^{\,\prime}}.$ If $\zeta_0\in
D^{\,\prime},$ the desired conclusion is obvious. Now, assume that
$\zeta_0\in
\partial D^{\,\prime}.$ We choose $\zeta_m^*\in D^{\,\prime}$
such that $d^{\,\prime}(\zeta_m,\zeta_m^*)<1/m$ and
$m_P(\overline{f^{\,-1}}(\zeta_m),
\overline{f^{\,-1}}(\zeta_m^*))<1/m,$ where $m_P$ is the metric
defined in the proposition \ref{pr2}. Since
$\zeta_m^*\stackrel{d^{\,\prime}}{\rightarrow} \zeta_0,$  we obtain
that $\overline{f^{\,-1}}(\zeta_m^*)\rightarrow
\overline{f^{\,-1}}(\zeta_0)$ as $m\rightarrow\infty.$ Thus,
$\overline{f^{\,-1}}(\zeta_m)\rightarrow
\overline{f^{\,-1}}(\zeta_0),$ as required.~$\Box$
\end{proof}

\medskip
{\bf Example.} Given $n\geqslant 2,$ $p\geqslant 1$ and $\alpha\in
\left(0, n/p(n-1)\right),$ set
$$
f(x)=\frac{1+|x|^{\alpha}}{|x|}\cdot x\,,\quad x\in {\Bbb
B}^n\setminus\{0\}\,.$$
It is not difficult to see that $f$ is a ring $Q$-homeomorphism of
${\Bbb B}^n\setminus\{0\}$ onto $A:=\{1<|y|<2\},$ where
$Q(x):=\left(\frac{1+r^{\,\alpha}}{\alpha
r^{\,\alpha}}\right)^{n-1},$ $r=|x|$ (see, e.g.,
\cite[Proposition~6.3]{MRSY}).  Moreover, $Q\in L^p({\Bbb B}^n).$ It
is clear that $A$ has a locally quasiconformal boundary, so, we can
consider that all prime ends in $A$ are single points of $\partial
A$ (see~\cite{Na}). Observe that $f$ has no continuous extension at
0, however, the inverse mapping
$f^{\,-1}(y)=\frac{y}{|y|}(|y|-1)^{1/\alpha}$ is continuous in
$\overline{A}.$ In particular, $f^{\,-1}({\Bbb S}^{n-1})=0.$ Thus,
the statement of the Theorem~\ref{th1} is not valid for $f,$ but is
valid for $f^{\,-1}.$ In this case, $Q\not\in FMO(0).$

\medskip
Combining Theorem \ref{th4} with Lemma \ref{lem2}, we obtain the
following statement.

\medskip
\begin{theorem}\label{th3}
{\sl\, Let $D$ and $D^{\,\prime}$ be domains with finite Hausdorff
dimensions $\alpha$ and $\alpha^{\,\prime}\geqslant 2$ in spaces
$(X,d,\mu)$ and $(X^{\,\prime},d^{\,\prime}, \mu^{\,\prime}),$
respectively. Assume that $X$ is complete and supports an
$\alpha$-Poincare inequality, and that the measure $µ$ is
doubling. Let $D$ be a bounded domain which is finitely connected
at the boundary, and let $Q:X\rightarrow (0, \infty)$ be an
integrable function in $D.$ Suppose that $f:D\rightarrow
D^{\,\prime},$ $D^{\,\prime}=f(D),$ is a ring $Q$-homeomorphism in
$\partial D.$ Moreover, suppose that $\partial D^{\,\prime}$ is
weakly flat and $\overline{D^{\,\prime}}$ is compact in
$X^{\,\prime}.$ Then $f$ has a homeomorphic extension
$f:\overline{D}_P\rightarrow \overline{D^{\,\prime}},$
$f(\overline{D}_P)=\overline{D^{\,\prime}},$ whenever $Q\in
FMO(\partial D).$ }
\end{theorem}

\section{Equicontinuity of families of homeomorphisms}

Now we prove that the corresponding families of ring
$Q$-homeomorphisms are equicontinuous in $\overline{D}_P=D\cup E_D,$
where $E_D$ is a prime ends space. In this section, we restrict us
by a case of homeomorphisms, only. Let us recall some definitions.
Let $(X, d)$ and $\left(X^{\,\prime}, d^{\,\prime}\right)$ be metric
spaces with distances $d$ and $d^{\,\prime}$, respectively. A family
$\frak{F}$ of mappings $f:X\rightarrow X^{\,\prime}$  is said to be
{\it equicontinuous at a point} $x_0 \in X$ if for every
$\varepsilon > 0$ there is $\delta
> 0$ such that $d^{\,\prime} (f(x),f(x_0))<\varepsilon$ for all $f
\in \frak{F}$ and $x \in X$ with $d(x, x_0)<\delta$. The family
$\frak{F}$ is {\it equicontinuous} if $\frak{F}$ is equicontinuous
at every point $x_0 \in X.$ In what follows, $X=\overline{D}_P$
and $d=m_P,$ where $m_P$ is defined in Proposition \ref{pr2}. The
next definition can be found, e.g., in \cite{NP}. A domain $D$ is
called a {\it uniform} domain if, for each $r>0,$ there is
$\delta>0$ such that $M_{\alpha}(\Gamma(F, F^{\,*},
D))\geqslant\delta$ whenever $F$ and $F^{\,*}$ are continua of $D$
with $d(F)\geqslant r$ and $d(F^{\,*})\geqslant r.$ Domains $D_i,$
$i\in I,$ are said to be {\it equi-uniform} domains if, for $r>0,$
the modulus condition above is satisfied by each $D_i$ with the
same number $\delta.$

\medskip
Given $\delta> 0,$ $D\subset X$ and a measurable function
$Q:D\rightarrow[0, \infty],$ denote $\frak{R}_{Q, \delta}(D)$ the
family of all ring $Q$-homeomorphisms $f:D\rightarrow
X^{\,\prime}\setminus K_f$ in $D,$ such that $f(D)$ is some open
set in $X^{\,\prime}$ and $d^{\,\prime}(K_f)=\sup\limits_{x,y\in
K_f}d^{\,\prime}(x, y)\geqslant \delta,$ where $K_f\subset
X^{\,\prime}$ is a continuum. The following statement holds.

\medskip
\begin{lemma}\label{lem4}
{\sl Let $(X,d,\mu)$ and $\left(X^{\,\prime},d^{\,\prime},
\mu^{\,\prime}\right)$ be metric spaces, let $D$ be a domain in $X$
with finite Hausdorff dimension $\alpha\geqslant 2,$ and let
$X^{\,\prime}$ be a domain with finite Hausdorff dimension
$\alpha^{\,\prime}\geqslant 2.$ Given $x_0\in D,$ assume that, there
exists a Lebesgue measurable function $\psi:(0, \infty)\rightarrow
(0, \infty)$ such that
%
%
%
%
$$I(\varepsilon,
\varepsilon_0):=\int\limits_{\varepsilon}^{\varepsilon_0}\psi(t)dt
< \infty$$
%
%
%
for every $\varepsilon\in (0,\varepsilon_0)$ and $I(\varepsilon,
\varepsilon_0)\rightarrow\infty$ as $\varepsilon\rightarrow 0,$ and

\begin{equation} \label{eq21}
\int\limits_{\varepsilon<d(x,
x_0)<\varepsilon_0}Q(x)\cdot\psi^{\,\alpha}(d(x, x_0)) \
d\mu(x)\,=\,o\left(I^{\,\alpha}(\varepsilon, \varepsilon_0)\right)
\end{equation}
as $\varepsilon\rightarrow 0.$ If $X$ is locally path connected and
locally compact space, and $X^{\,\prime}$ is a uniform domain, then
$\frak{R}_{Q, \delta}(D)$ is equicontinuous at $x_0.$}
\end{lemma}

\medskip
\begin{proof}
The idea of a proof is closely related to \cite[Lemma~2]{Sev$_2$}.
Assume the contrary, i.e., assume that $\frak{R}_{Q, \delta}(D)$ is
not equicontinuous at $x_0.$ Now, there is exists $x_k\in D$ and
$f_k\in\frak{R}_{Q, \delta}(D)$ such that $x_k\rightarrow x_0$ as
$k\rightarrow\infty$ and
\begin{equation}\label{eq22}
d^{\,\prime}(f_k(x_k), f_k(x_0))\geqslant\varepsilon_0
\end{equation}
for some $\varepsilon_0.$ Since $X$ is locally connected by
assumption, there is a sequence of balls $B(x_0, \varepsilon_k),$
$k=0,1,2,\ldots,$ $\varepsilon_k\rightarrow 0$ as
$k\rightarrow\infty,$ such that $V_{k+1}\subset \overline{B(x_0,
\varepsilon_k)}\subset V_k,$ where the $V_k$ are continua in $D.$
There is no loss of generality in assuming that $x_k\in V_k.$ Now,
$x_0$ and $x_k$ can be joined by a curve $\gamma_k$ in the domain
$V_k.$ Note that an arbitrary curve $\gamma\in \Gamma(K_{f_k},
f_k(\gamma_k), X^{\,\prime})$ is not included entirely both in
$f_k(B(x_0, \varepsilon_0))$ and $X^{\,\prime}\setminus f_k(B(x_0,
\varepsilon_0)),$ therefore there exists $y_1\in |\gamma|\cap
f_k(S(x_0, \varepsilon_0))$ (see \cite[Theorem 1, $\S\,$46, item
I]{Ku}). Let $\gamma:[0, 1]\rightarrow X^{\,\prime}$ and let
$t_1\in (0, 1)$ be such that $\gamma(t_1)=y_1.$ There is no loss
of generality in assuming that $|\gamma|_{[0, t_1)}|\subset
f_k(B(x_0, \varepsilon_0)).$ We put $\gamma_1:=\gamma|_{[0,
t_1)},$ and $\alpha_1=f_k^{\,-1}(\gamma_1).$ Observe that
$|\alpha_1|\subset B(x_0, \varepsilon_0),$ moreover, $\alpha_1$ is
not included entirely either in $\overline{B(x_0,
\varepsilon_{k-1})},$ or in $X\setminus\overline{B(x_0,
\varepsilon_{k-1})}.$ Consequently, there exists $t_2\in (0, t_1)$
with $\alpha_1(t_2)\in S(x_0, \varepsilon_{k-1})$ (see
\cite[Theorem 1, $\S\,$46, item I]{Ku}). There is no loss of
generality in assuming that $|\alpha_1|_{[t_2,\,t_1]}|\subset
X\setminus\overline{B(x_0, \varepsilon_{k-1})}.$ Put
$\alpha_2=\alpha_1|_{[t_2,\,t_1]}.$ Observe that
$\gamma_2:=f_k(\alpha_2)$ is a subcurve of $\gamma.$ By the said
above,
$$\Gamma(K_{f_k}, f_k(\gamma_k),
X^{\,\prime})>\Gamma(f_k(S(x_0, \varepsilon_{k-1})), f_k(S(x_0,
\varepsilon_0)), f_k(A))\,,$$
where $A=A(x_0, \varepsilon_{k-1}, \varepsilon_0),$ and by
(\ref{eq32*A}) we obtain
\begin{equation}\label{eq3B}
M_{\alpha^{\,\prime}}(\Gamma(K_{f_k}, f_k(\gamma_k),
X^{\,\prime}))\leqslant M_{\alpha^{\,\prime}}(\Gamma(f_k(S(x_0,
\varepsilon_{k-1})), f_k(S(x_0, \varepsilon_0)), f_k(A)))\,.
\end{equation}
Since $I(\varepsilon, \varepsilon_0)\rightarrow\infty$ as
$\varepsilon\rightarrow 0,$ we can consider that $I(\varepsilon_k,
\varepsilon_0)>0$ for every $k=1,2,\ldots\,.$ Consider the family
of measurable functions
$$\eta_k(t)=\psi(t)/I(\varepsilon_k,
\varepsilon_0), \quad t\in(\varepsilon,\, \varepsilon_0)\,.$$
Observe that
$\int\limits_{\varepsilon_k}^{\varepsilon_0}\eta_k(t)\,dt=1.$ Now,
by (\ref{eq1C}), (\ref{eq21}) and (\ref{eq3B}), we obtain that
\begin{equation}\label{eq23}
M_{\alpha^{\,\prime}}(\Gamma(K_{f_k}, f_k(\gamma_k),
X^{\,\prime}))\leqslant \varphi(\varepsilon_k)\,,
\end{equation}
where $\varphi$ is some function with
$\varphi(\varepsilon_k)\rightarrow 0$ as $k\rightarrow\infty.$ From
other hand, it follows from (\ref{eq22}) that
$\min\{d^{\,\prime}(K_{f_k}), d^{\,\prime}(f_k(\gamma_k))\}\geqslant
r_0$ for some $r_0>0$ and every $k=1,2,\ldots,.$ Now, since
$X^{\,\prime}$ is uniformly domain, we obtain that
\begin{equation}\label{eq24}
M_{\alpha^{\,\prime}}(\Gamma(K_{f_k}, f_k(\gamma_k),
X^{\,\prime}))\geqslant \delta_0
\end{equation}
for some $\delta_0>0$ and every $k=1,2,\ldots, .$ Now, (\ref{eq24})
contradicts with (\ref{eq23}). Thus, $\frak{R}_{Q, \delta}(D)$ is
equicontinuous at $x_0,$ as required.~$\Box$
\end{proof}

\medskip
Given $\delta> 0,$ $D\subset X,$ a continuum $A\subset D$ and a
measurable function $Q:D\rightarrow[0, \infty],$ denote
$\frak{F}_{Q, \delta, A}(D)$ the family of all ring
$Q$-homeomorphisms $f:D\rightarrow X^{\,\prime}\setminus K_f$ in
$D,$ such that such that $f(D)$ is some open set in $X^{\,\prime}$
and $d^{\,\prime}(K_f)=\sup\limits_{x,y\in K_f}d^{\,\prime}(x,
y)\geqslant \delta$ and $d^{\,\prime}(f(A))\geqslant \delta,$ where
$K_f\subset X^{\,\prime}$ is a continuum. An analog of a following
result was proved in \cite[Theorem~3.1]{NP}.

\medskip
\begin{lemma}\label{lem3}
{\sl Let $D$ and $D_f^{\,\prime}:=f(D),$ $f\in\frak{F}_{Q, \delta,
A}(D),$ be domains with finite Hausdorff dimensions $\alpha$ and
$\alpha^{\,\prime}\geqslant 2$ in spaces $(X,d,\mu)$ and
$(X^{\,\prime},d^{\,\prime}, \mu^{\,\prime}),$ respectively, and let
$X^{\,\prime}$ be a domain with finite Hausdorff dimension
$\alpha^{\,\prime}\geqslant 2.$ Assume that $X$ is complete and
supports an $\alpha$-Poincare inequality, and that the measure $µ$
is doubling. Let $D$ be a bounded domain which is finitely connected
at the boundary, and let $Q:X\rightarrow (0, \infty)$ be a locally
integrable function. Assume that, for every $x_0\in  \overline{D},$
there exists a Lebesgue measurable function $\psi:(0,
\infty)\rightarrow (0, \infty)$ such that
%
%
%
%
$$I(\varepsilon,
\varepsilon_0):=\int\limits_{\varepsilon}^{\varepsilon_0}\psi(t)dt
< \infty$$
%
%
%
for every $\varepsilon\in (0,\varepsilon_0)$ and $I(\varepsilon,
\varepsilon_0)\rightarrow\infty$ as $\varepsilon\rightarrow 0,$ and

\begin{equation} \label{eq19}
\int\limits_{\varepsilon<d(x,
x_0)<\varepsilon_0}Q(x)\cdot\psi^{\,\alpha}(d(x, x_0)) \
d\mu(x)\,=\,o\left(I^{\,\alpha}(\varepsilon, \varepsilon_0)\right)
\end{equation}
as $\varepsilon\rightarrow 0.$  If $D_f^{\,\prime}:=f(D)$ and
$X^{\,\prime}$ are equi-uniform domains over $f\in \frak{F}_{Q,
\delta, A}(D)$ and $\overline{D_f^{\,\prime}}$ are compacts in
$X^{\,\prime},$ then every $f\in\frak{F}_{Q, \delta, A}(D)$ has a
continuous extension $f:\overline{D}_P\rightarrow
\overline{D_f^{\,\prime}},$ and $\frak{F}_{Q, \delta, A}(D)$ is
equicontinuous in $\overline{D}_P.$}
\end{lemma}

\medskip
\begin{proof}
Observe that $\partial D_f^{\,\prime}=\partial f(D)$ is strongly
accessible for every $f\in \frak{F}_{Q, \delta, A}(D).$ Indeed,
assume that $x_0\in \partial D^{\,\prime}.$ Given a neighborhood $U$
of $x_0,$ there exists $\varepsilon_1>0$ such that
$V:=\overline{B(x_0, \varepsilon_1)}\subset U.$ Assume that
$\partial U\ne\varnothing$ and $\partial V\ne\varnothing,$ now
$\varepsilon_2:=d^{\,\prime}(\partial U, \partial V)>0.$ Since
$D_f^{\,\prime}$ are equi-uniform, we obtain that
$d^{\,\prime}(F)\geqslant \varepsilon_2$ and
$d^{\,\prime}(G)\geqslant \varepsilon_2$ whenever $F$ and $G$ are
continua in $\partial D^{\,\prime}$ with $F\cap\partial
U\ne\varnothing\ne F\cap\partial V$ and $G\cap\partial
U\ne\varnothing\ne G\cap\partial V.$ Thus, $\partial
D_f^{\,\prime}=f(D)$ is strongly accessible, as required. Now, by
Lemma \ref{lem1} every $f\in \frak{F}_{Q, \delta, A}(D)$ has a
continuous extension $f:\overline{D}_P\rightarrow
\overline{D_f^{\,\prime}}.$

\medskip
Since $\mu$ is doubling, $X$ is complete if and only if it is proper
(i.e. every closed bounded set is compact), see
\cite[Proposition~3.1]{BB}. Now, $X$ is a locally compact space.
Since $X$ is complete, $X$ supports an $\alpha$-Poincare inequality,
and the measure $µ$ is doubling, we obtain that $X$ is locally
connected (see \cite{ABBS}, see also \cite[Theorem~17.1]{Ch}).
Moreover, $X$ is locally path connected by the
Mazurkiewicz–Moore–Menger theorem (see in \cite[Theorem~1, Ch.~6,
$\S$ 50, item II]{Ku}. Thus, all conditions of Lemma \ref{lem4} are
satisfied. Now, by Lemma \ref{lem4}, $\frak{F}_{Q, \delta, A}(D)$ is
equicontinuous at $x_0$ for every $x_0\in D.$

\medskip
It remains to show that $\frak{F}_{Q, \delta, A}(D)$ is
equicontinuous on $E_D=\overline{D}_P\setminus D.$ Assume the
contrary, i.e., assume that there exists $P_0\in E_D$ such that
$\frak{F}_{Q, \delta, A}(D)$ is not equicontinuous at $P_0.$ Now,
there is exists $P_k\in \overline{D}_P$ and $f_k\in\frak{F}_{Q,
\delta, A}(D)$ such that $P_k\rightarrow P_0$ as
$k\rightarrow\infty$ and
\begin{equation}\label{eq25}
d^{\,\prime}(f_k(P_k), f_k(P_0))\geqslant\varepsilon_0
\end{equation}
for some $\varepsilon_0.$ Since $f_k$ has a continuous extension on
$\overline{D}_P,$ given $k\in {\Bbb N},$ we can find $x_k\in D$ with
$m_P(x_k, P_k)<1/k$ and $d(f_k(x_k), f_k(P_k))<1/k.$ Thus, we obtain
from (\ref{eq25}) that
\begin{equation}\label{eq4C}
d^{\,\prime}(f_k(x_k), f_k(P_0))\geqslant
\varepsilon_0/2\quad\forall\quad k=1,2,\ldots,\,.
\end{equation}
Similarly, we can find $x_k^{\,\prime}\in D$ such that
$x_k^{\,\prime}\rightarrow P_0$ as $k\rightarrow \infty,$ and
$d^{\,\prime}(f_k(x_k^{\,\prime}), f_k(P_0))<1/k,$ $k=1,2,\ldots\,.$
Thus, we obtain from (\ref{eq4C}) that
\begin{equation}\label{eq5E}
d^{\,\prime}(f_k(x_k), f_k(x_k^{\,\prime}))\geqslant
\varepsilon_0/4\quad\forall\quad k=1,2,\ldots\,,
\end{equation}
where $x_k$ and $x_k^{\,\prime}\in D$ satisfy conditions
$x_k\rightarrow P_0,$ $x^{\,\prime}_k\rightarrow P_0$ as
$k\rightarrow\infty.$

\medskip
Denote $x_0:=I([E_k])$ (see Proposition \ref{pr3}). By Remark 4.5 in
\cite{ABBS} we can consider that the sets $E_k$ are open. Moreover,
by Remark 2.6 in \cite{ABBS} the set $E_k$ is path connected for
every $k\in {\Bbb N}.$ Arguing as in the proof of Lemma \ref{lem1},
we can show that, for every $r>0$ there exists $k\in {\Bbb N}$ such
that $E_k\subset B(x_0, r)\cap D.$ Thus, there is no loss of
generality in assuming that $x_k, x_k^{\,\prime}\in E_k$ and
$E_k\subset B(x_0, 2^{\,-k}).$ Let $\gamma_k$ be a path, joining
$x_k$ and $x_k^{\,\prime}$ in $E_k.$ Observe that $A\subset
D\setminus B(x_0, 2^{\,-k})$ for all $k>k_0$ and some $k_0\in {\Bbb
N}.$ We can consider that $2^{\,-k_0}<\varepsilon_0.$ Let $\Gamma_k$
be a family of curves joining $\gamma_k$ and $A$ in
$D^{\,\prime}_{f_k}.$ By Remark \ref{rem1}, we obtain that
\begin{equation}\label{eq26}
M_{\alpha^{\,\prime}}(f_k(\Gamma_k))\leqslant
M_{\alpha^{\,\prime}}(f_k(\Gamma(S(x_0, 2^{\,-k}), S(x_0,
2^{\,-k_0}), A(x_0, 2^{\,-k}, 2^{\,-k_0}))))\,.
\end{equation}
Observe that
$$\eta(t)=\left\{
\begin{array}{rr}
\psi(t)/I(2^{\,-k}, 2^{\,-k_0}), &   t\in (2^{\,-k},
2^{\,-k_0}),\\
0,  &  t\in {\Bbb R}\setminus (2^{\,-k}, 2^{\,-k_0})\,,
\end{array}
\right. $$
$I(\varepsilon,
\varepsilon_0):=\int\limits_{\varepsilon}^{\varepsilon_0}\psi(t)dt,$
satisfies the condition (\ref{eq*3!!}) at $r_1=2^{-k}$ and
$r_2=2^{-k_0}.$ By the definition of a ring $Q$-homeomorphism at a
boundary point, (\ref{eq19}) and (\ref{eq26}) imply
\begin{equation}\label{eq11A}
M_p(f_k(\Gamma_k))\leqslant \alpha(2^{\,-k})\rightarrow 0
\end{equation}
as $k\rightarrow \infty,$ where $\alpha(\varepsilon)$ is some
nonnegative function with $\alpha(\varepsilon)\rightarrow 0$ as
$\varepsilon\rightarrow 0.$

\medskip
However, $f_k(\Gamma_k)=\Gamma(f_k(\gamma_k), f_k(A),
D_{f_k}^{\,\prime}).$ By assumption,
$d^{\,\prime}(f_k(A))\geqslant \delta,$ $k=1,2,\ldots,$ moreover,
by (\ref{eq5E}) we obtain that
$d^{\,\prime}(f_k(\gamma_k))\geqslant \varepsilon_0/4,$
$k=1,2,\ldots, .$ Since $D_{f_k}^{\,\prime}$ are are equi-uniform
domains, we obtain that
\begin{equation}\label{eq27}
M_{\alpha^{\,\prime}}(f_k(\Gamma_k))\geqslant r_0\,,k=1,2,\ldots\,,
\end{equation}
for some $r_0>0.$ But (\ref{eq27}) contradicts (\ref{eq11A}). Thus,
$\frak{F}_{Q, \delta, A}(D)$ is equicontinuous at $P_0,$ as
required.~$\Box$
\end{proof}

\medskip
The following main result holds.

\medskip
\begin{theorem}\label{th5}
{\sl Let $D$ and $D_f^{\,\prime}:=f(D),$ $f\in\frak{F}_{Q, \delta,
A}(D),$ be domains with finite Hausdorff dimensions $\alpha$ and
$\alpha^{\,\prime}\geqslant 2$ in spaces $(X,d,\mu)$ and
$(X^{\,\prime},d^{\,\prime}, \mu^{\,\prime}),$ respectively, and let
$X^{\,\prime}$ be a domain with finite Hausdorff dimension
$\alpha^{\,\prime}\geqslant 2.$ Assume that $X$ is complete and
supports an $\alpha$-Poincare inequality, and that the measure $µ$
is doubling. Let $D$ be a bounded domain which is finitely connected
at the boundary, and let $Q:X\rightarrow (0, \infty)$ be a locally
integrable function. Assume that, $Q\in FMO(\overline{D}).$ If
$D_f^{\,\prime}:=f(D)$ and $X^{\,\prime}$ are equi-uniform domains
over $f\in \frak{F}_{Q, \delta, A}(D)$ and
$\overline{D_f^{\,\prime}}$ are compacts in $X^{\,\prime},$ then
$\frak{F}_{Q, \delta, A}(D)$ is equicontinuous in $\overline{D}_P.$}
\end{theorem}

\medskip
{\it Proof of the Theorem~\ref{th5}} follows from Lemma~\ref{lem3}
and Proposition~\ref{pr3A}. Indeed, $X$ is upper regular by
(\ref{eq2}), and (\ref{eq7}) holds because the measure $µ$ is
doubling by assumptions. So, the desired statement follows from
the Lemma~\ref{lem3}.~$\Box$

\section{Equicontinuity of families of maps with \textbf{A}-condition}

\medskip
Given $\delta> 0,$ $D\subset X$ and a measurable function
$Q:D\rightarrow[0, \infty],$ denote $\frak{G}_{Q, \delta,
\textbf{A}}(D)$ the family of all open discrete ring $Q$-maps
$f:D\rightarrow X^{\,\prime}\setminus K_f$ in $D$ with
\textbf{A}-condition, such that
$d^{\,\prime}(K_f)=\sup\limits_{x,y\in K_f}d^{\,\prime}(x,
y)\geqslant \delta,$ where $K_f\subset X^{\,\prime}$ is a continuum.
The following statement holds.

\medskip
\begin{lemma}\label{lem5}
{\sl Let $(X,d,\mu)$ and $\left(X^{\,\prime},d^{\,\prime},
\mu^{\,\prime}\right)$ be metric spaces, let $D$ be a domain in $X$
with finite Hausdorff dimension $\alpha\geqslant 2,$ and let
$X^{\,\prime}$ be a domain with finite Hausdorff dimension
$\alpha^{\,\prime}\geqslant 2.$ Given $x_0\in D,$ assume that, there
exists a Lebesgue measurable function $\psi:(0, \infty)\rightarrow
(0, \infty)$ such that
%
%
%
%
$$I(\varepsilon,
\varepsilon_0):=\int\limits_{\varepsilon}^{\varepsilon_0}\psi(t)dt
< \infty$$
%
%
%
for every $\varepsilon\in (0,\varepsilon_0)$ and $I(\varepsilon,
\varepsilon_0)\rightarrow\infty$ as $\varepsilon\rightarrow 0,$ and

\begin{equation} \label{eq29}
\int\limits_{\varepsilon<d(x,
x_0)<\varepsilon_0}Q(x)\cdot\psi^{\,\alpha}(d(x, x_0)) \
d\mu(x)\,=\,o\left(I^{\,\alpha}(\varepsilon, \varepsilon_0)\right)
\end{equation}
as $\varepsilon\rightarrow 0.$ If $X$ is locally path connected and
locally compact space, and $X^{\,\prime}$ is a uniform domain, then
$\frak{R}_{Q, \delta}(D)$ is equicontinuous at $x_0.$}
\end{lemma}

\medskip
\begin{proof}
The idea of a proof is closely related to \cite[Lemma~2]{Sev$_2$},
and similar to Lemma \ref{lem4}. Assume the contrary, i.e., assume
that $\frak{G}_{Q, \delta, \textbf{A}}(D)$ is not equicontinuous
at $x_0.$ Now, there exists $x_k\in D$ and $f_k\in\frak{G}_{Q,
\delta, \textbf{A}}(D)$ such that $x_k\rightarrow x_0$ as
$k\rightarrow\infty$ and
\begin{equation}\label{eq30}
d^{\,\prime}(f_k(x_k), f_k(x_0))\geqslant\varepsilon_0
\end{equation}
for some $\varepsilon_0.$ Since $X$ is locally compact metric space,
we can consider that $\overline{B(x_0, \varepsilon_0)}$ is a compact
set in $X.$ Since $X$ is locally compact metric space, we can
consider that $\overline{B(x_0, \varepsilon_0)}$ is a compact set in
$X.$ Since $X$ is locally connected by assumption, there is a
sequence of balls $B(x_0, \varepsilon_k),$ $k=0,1,2,\ldots,$
$\varepsilon_k\rightarrow 0$ as $k\rightarrow\infty,$ such that
$V_{k+1}\subset \overline{B(x_0, \varepsilon_k)}\subset V_k,$ where
the $V_k$ are continua in $D.$ There is no loss of generality in
assuming that $x_k\in V_k.$ Now, $x_0$ and $x_k$ can be joined by a
curve $\gamma_k$ in the domain $V_k.$

\medskip
By \cite[Lemma~3]{Sev$_2$}, (\ref{eq29}) implies that
\begin{equation}\label{eq31}
M_{\alpha^{\,\prime}}(\Gamma(f_k(\overline{B(x_0, \varepsilon)}),
\partial f_k(B(x_0, \varepsilon_0)), X^{\,\prime}))\leqslant
\alpha(\varepsilon)
\end{equation}
as $\varepsilon\rightarrow 0,$ where $\alpha(\varepsilon)$ is some
function with $\alpha(\varepsilon)\rightarrow 0$ as
$\varepsilon\rightarrow 0.$ Thus, we obtain from (\ref{eq31}) that
\begin{equation}\label{eq32}
M_{\alpha^{\,\prime}}(\Gamma(f_k(\gamma_k),
\partial f_k(B(x_0, \varepsilon_0)), X^{\,\prime}))\leqslant
\alpha(\varepsilon_{k-1})\rightarrow 0,\quad k\rightarrow\infty\,.
\end{equation}
From other hand, observe that $\Gamma(K_{f_k}, f_k(\gamma_k),
X^{\,\prime})>\Gamma(f_k(\gamma_k),
\partial f_k(B(x_0, \varepsilon_0)), X^{\,\prime})$ (see \cite[Ch.~5, $\S\,$46,
item I]{Ku}); consequently, by (\ref{eq32*A}) we obtain
\begin{equation}\label{eq9B}
M_{\alpha^{\,\prime}}(\Gamma(K_{f_k}, f_k(\gamma_k),
X^{\,\prime}))\leqslant M_{\alpha^{\,\prime}}(\Gamma(f_k(\gamma_k),
\partial f_k(B(x_0, \varepsilon_0)), X^{\,\prime}))\,.
\end{equation}
By (\ref{eq30}), we obtain that
$d^{\,\prime}(f_k(\gamma_k))\geqslant\varepsilon_0$ for every
$k=1,2,\ldots,,$ moreover, $d^{\,\prime}(K_{f_k})\geqslant\delta$
for every $k=1,2,\ldots,$ by assumption of the lemma. Now, since
$X^{\,\prime}$ is a uniform domain, we obtain that
\begin{equation}\label{eq33}
M_{\alpha^{\,\prime}}(\Gamma(K_{f_k}, f_k(\gamma_k),
X^{\,\prime}))\geqslant r_0
\end{equation}
for each $k=1,2,\ldots,$ and some $r_0>0.$ Observe that
(\ref{eq33}) contradicts with (\ref{eq32}) and (\ref{eq9B}). Thus,
$\frak{G}_{Q, \delta, \textbf{A}}(D)$ is equicontinuous at $x_0,$
as required.~$\Box$
\end{proof}

\medskip
Let $\delta> 0,$ let $D\subset X$ and let $Q:D\rightarrow[0,
\infty]$ be a measurable function. Denote $\frak{E}_{Q, \delta,
\textbf{A}}(D)$ the family of all open closed discrete ring $Q$-maps
$f:D\rightarrow X^{\,\prime}$ with the following conditions: 1) $f$
satisfies \textbf{A}-condition in $D;$ 2) given $f:D\rightarrow
X^{\,\prime}$ there exists a continuum $K_f\subset
X^{\,\prime}\setminus f(D)$ and
$d^{\,\prime}(K_f)=\sup\limits_{x,y\in K_f}d^{\,\prime}(x,
y)\geqslant \delta;$ 3) given $f:D\rightarrow X^{\,\prime}$ there
exists a continuum $A_f\subset f(D)$ such that
$d^{\,\prime}(A_f)\geqslant\delta$ and $d(f^{\,-1}(A_f), \partial
D)\geqslant\delta.$ The following statement holds.

\medskip
\begin{lemma}\label{lem6}
{\sl Let $D$ and $D_f^{\,\prime}:=f(D),$ $\frak{E}_{Q, \delta,
\textbf{A}}(D),$ be domains with finite Hausdorff dimensions
$\alpha$ and $\alpha^{\,\prime}\geqslant 2$ in spaces $(X,d,\mu)$
and $(X^{\,\prime},d^{\,\prime}, \mu^{\,\prime}),$ respectively, and
let $X^{\,\prime}$ be a domain with finite Hausdorff dimension
$\alpha^{\,\prime}\geqslant 2.$ Assume that $X$ is complete and
supports an $\alpha$-Poincare inequality, and that the measure $µ$
is doubling. Let $D$ be a bounded domain which is finitely connected
at the boundary, and let $Q:X\rightarrow (0, \infty)$ be a locally
integrable function. Assume that, for every $x_0\in  \overline{D},$
there exists a Lebesgue measurable function $\psi:(0,
\infty)\rightarrow (0, \infty)$ such that
%
%
%
%
$$I(\varepsilon,
\varepsilon_0):=\int\limits_{\varepsilon}^{\varepsilon_0}\psi(t)dt
< \infty$$
%
%
%
for every $\varepsilon\in (0,\varepsilon_0)$ and $I(\varepsilon,
\varepsilon_0)\rightarrow\infty$ as $\varepsilon\rightarrow 0,$ and

\begin{equation} \label{eq35}
\int\limits_{\varepsilon<d(x,
x_0)<\varepsilon_0}Q(x)\cdot\psi^{\,\alpha}(d(x, x_0)) \
d\mu(x)\,=\,o\left(I^{\,\alpha}(\varepsilon, \varepsilon_0)\right)
\end{equation}
as $\varepsilon\rightarrow 0.$  If $D_f^{\,\prime}:=f(D)$ and
$X^{\,\prime}$ are equi-uniform domains over $\frak{E}_{Q, \delta,
\textbf{A}}(D)$ and $\overline{D_f^{\,\prime}}$ are compacts in
$X^{\,\prime},$ then every $f\in\frak{E}_{Q, \delta, \textbf{A}}(D)$
has a continuous extension $f:\overline{D}_P\rightarrow
\overline{D_f^{\,\prime}},$ and $\frak{E}_{Q, \delta,
\textbf{A}}(D)$ is equicontinuous in $\overline{D}_P.$}
\end{lemma}

\medskip
\begin{proof}
Arguing as in the proof of Lemma \ref{lem3}, we obtain that
$\partial D_f^{\,\prime}=\partial f(D)$ is strongly accessible for
every $f\in\frak{E}_{Q, \delta, \textbf{A}}(D).$ Moreover, we see
that $X$ is a locally compact and locally path connected space. Now,
by Lemma \ref{lem1} every $f\in \frak{F}_{Q, \delta, A}(D)$ has a
continuous extension $f:\overline{D}_P\rightarrow
\overline{D_f^{\,\prime}}.$ By Lemma \ref{lem5}, we also obtain that
$\frak{E}_{Q, \delta, \textbf{A}}(D)$ is equicontinuous in $D,$
because $\frak{E}_{Q, \delta, \textbf{A}}(D)\subset \frak{G}_{Q,
\delta, \textbf{A}}(D).$

\medskip
It remains to show that $\frak{E}_{Q, \delta, \textbf{A}}(D)$ is
equicontinuous on $E_D:=\overline{D}_P\setminus D.$ Assume the
contrary, i.e., assume that there exists $P_0\in E_D$ such that
$\frak{E}_{Q, \delta, \textbf{A}}(D)$ is not equicontinuous at
$P_0.$ Now, there is exists $P_k\in \overline{D}_P$ and
$f_k\in\frak{E}_{Q, \delta, \textbf{A}}(D)$ such that
$P_k\rightarrow P_0$ as $k\rightarrow\infty$ and
\begin{equation}\label{eq36}
d^{\,\prime}(f_k(P_k), f_k(P_0))\geqslant\varepsilon_0
\end{equation}
for some $\varepsilon_0.$ Since $f_k$ has a continuous extension on
$\overline{D}_P,$ given $k\in {\Bbb N},$ we can find $x_k\in D$ with
$m_P(x_k, P_k)<1/k$ and $d(f_k(x_k), f_k(P_k))<1/k.$ Thus, we obtain
from (\ref{eq36}) that
\begin{equation}\label{eq37}
d^{\,\prime}(f_k(x_k), f_k(P_0))\geqslant
\varepsilon_0/2\quad\forall\quad k=1,2,\ldots,\,.
\end{equation}
Similarly, we can find $x_k^{\,\prime}\in D$ such that
$x_k^{\,\prime}\rightarrow P_0$ as $k\rightarrow \infty,$ and
$d^{\,\prime}(f_k(x_k^{\,\prime}), f_k(P_0))<1/k,$ $k=1,2,\ldots\,.$
Thus, we obtain from (\ref{eq37}) that
%
$$d^{\,\prime}(f_k(x_k), f_k(x_k^{\,\prime}))\geqslant
\varepsilon_0/4\quad\forall\quad k=1,2,\ldots\,,$$
%
where $x_k$ and $x_k^{\,\prime}\in D$ satisfy conditions
$x_k\rightarrow P_0,$ $x^{\,\prime}_k\rightarrow P_0$ as
$k\rightarrow\infty.$

\medskip
Denote $x_0:=I([E_k])$ (see Proposition \ref{pr3}). By Remark 4.5 in
\cite{ABBS} we can consider that the sets $E_k$ are open. Moreover,
by Remark 2.6 in \cite{ABBS} the set $E_k$ is path connected for
every $k\in {\Bbb N}.$ Arguing as in the proof of Lemma \ref{lem1},
we can show that, for every $r>0$ there exists $k\in {\Bbb N}$ such
that $E_k\subset B(x_0, r)\cap D.$ Thus, there is no loss of
generality in assuming that $x_k, x_k^{\,\prime}\in E_k$ and
$E_k\subset B(x_0, 2^{\,-k}).$ Let $\gamma_k$ be a path, joining
$x_k$ and $x_k^{\,\prime}$ in $E_k.$ Let $A_{f_k}$ be the set from
the definition of $\frak{E}_{Q, \delta, \textbf{A}}(D).$ Observe
that $f_k^{\,-1}(A_{f_k})\subset D\setminus B(x_0, 2^{\,-k})$ for
all $k>k_0$ and some $k_0\in {\Bbb N}.$ We can consider that
$2^{\,-k_0}<\varepsilon_0.$ Let $\Gamma_k$ be a family of curves
joining $\gamma_k$ and $f_k^{\,-1}(A_{f_k})$ in $D.$ By Remark
\ref{rem1}, we obtain that
\begin{equation}\label{eq38}
M_{\alpha^{\,\prime}}(f_k(\Gamma_k))\leqslant
M_{\alpha^{\,\prime}}(f_k(\Gamma(S(x_0, 2^{\,-k}), S(x_0,
2^{\,-k_0}), A(x_0, 2^{\,-k}, 2^{\,-k_0}))))\,.
\end{equation}
Observe that
$$\eta(t)=\left\{
\begin{array}{rr}
\psi(t)/I(2^{\,-k}, 2^{\,-k_0}), &   t\in (2^{\,-k},
2^{\,-k_0}),\\
0,  &  t\in {\Bbb R}\setminus (2^{\,-k}, 2^{\,-k_0})\,,
\end{array}
\right. $$
$I(\varepsilon,
\varepsilon_0):=\int\limits_{\varepsilon}^{\varepsilon_0}\psi(t)dt,$
satisfies the condition (\ref{eq*3!!}) at $r_1=2^{-k}$ and
$r_2=2^{-k_0}.$ By the definition of a ring $Q$-homeomorphism at a
boundary point, (\ref{eq35}) and (\ref{eq38}) imply
\begin{equation}\label{eq41}
M_{\alpha^{\,\prime}}(f_k(\Gamma_k))\leqslant
\alpha(2^{\,-k})\rightarrow 0
\end{equation}
as $k\rightarrow \infty,$ where $\alpha(\varepsilon)$ is some
nonnegative function with $\alpha(\varepsilon)\rightarrow 0$ as
$\varepsilon\rightarrow 0.$

\medskip
From other hand, let us consider the family $\Gamma(f_k(\gamma_k),
A_{f_k}, D^{\,\prime}_{f_k}).$ Since $D_{f_k}^{\,\prime}:=f_k(D)$
are equi-uniform domains, we obtain that
\begin{equation}\label{eq39}
M_{\alpha^{\,\prime}}(\Gamma(f_k(\gamma_k), A_{f_k},
D^{\,\prime}_{f_k}))\geqslant r_0\,,k=1,2,\ldots,
\end{equation}
for some $r_0>0.$
Let $\Gamma_k^*$ be the family of all maximal $f_k$-liftings of
$\Gamma(f_k(\gamma_k), A_{f_k}, D^{\,\prime}_{f_k})$ starting at
$\gamma_k.$ (The family $\Gamma_k^*$ is well defined in view of
condition \textbf{A}). Arguing as in the proof of Lemma~\ref{lem1},
we can show that $\Gamma_k^*\subset \Gamma_k.$ Besides that,
$f_k(\Gamma_k^{\,*})<\Gamma(f_k(\gamma_k), A_{f_k},
D^{\,\prime}_{f_k}).$ Thus, we obtain
\begin{equation}\label{eq40}
M_{\alpha^{\,\prime}}(\Gamma(f_k(\gamma_k), A_{f_k},
D^{\,\prime}_{f_k}))\leqslant
M_{\alpha^{\,\prime}}(f_k(\Gamma_k^{\,*}))\leqslant
M_{\alpha^{\,\prime}}(f_k(\Gamma_k))\,.
\end{equation}
But (\ref{eq39}) and (\ref{eq40}) contradict with (\ref{eq41}).
Thus, $\frak{E}_{Q, \delta, \textbf{A}}(D)$ is equicontinuous at
$P_0,$ as required.~$\Box$
\end{proof}

\medskip
The following main result holds.

\medskip
\begin{theorem}\label{th2}
{\sl Let $D$ and $D_f^{\,\prime}:=f(D),$ $\frak{E}_{Q, \delta,
\textbf{A}}(D),$ be domains with finite Hausdorff dimensions
$\alpha$ and $\alpha^{\,\prime}\geqslant 2$ in spaces $(X,d,\mu)$
and $(X^{\,\prime},d^{\,\prime}, \mu^{\,\prime}),$ respectively, and
let $X^{\,\prime}$ be a domain with finite Hausdorff dimension
$\alpha^{\,\prime}\geqslant 2.$ Assume that $X$ is complete and
supports an $\alpha$-Poincare inequality, and that the measure $µ$
is doubling. Let $D$ be a bounded domain which is finitely connected
at the boundary, and let $Q:X\rightarrow (0, \infty)$ be a locally
integrable function. Assume that $Q\in FMO(\overline{D}).$ If
$D_f^{\,\prime}:=f(D)$ and $X^{\,\prime}$ are equi-uniform domains
over $\frak{E}_{Q, \delta, \textbf{A}}(D)$ and
$\overline{D_f^{\,\prime}}$ are compacts in $X^{\,\prime},$ then
every $f\in\frak{E}_{Q, \delta, \textbf{A}}(D)$ has a continuous
extension $f:\overline{D}_P\rightarrow \overline{D_f^{\,\prime}},$
and $\frak{E}_{Q, \delta, \textbf{A}}(D)$ is equicontinuous in
$\overline{D}_P.$}
\end{theorem}

\medskip
{\it Proof of the Theorem~\ref{th2}} follows from Lemma~\ref{lem6}
and Proposition~\ref{pr3A}. Indeed, $X$ is upper regular by
(\ref{eq2}), and (\ref{eq7}) holds because the measure $µ$ is
doubling by assumptions. So, the desired statement follows from
the Lemma~\ref{lem6}.~$\Box$

\bigskip
{\bf Acknowledgments.} Author thanks Professor Tomasz Adamowicz,
Institute of Mathematics of Polish Academy of Science, Warsaw, for
joint discussion and useful recommendations.

\medskip
{\bf \noindent Evgeny Sevost'yanov} \\
Zhytomyr Ivan Franko State University,  \\
40 Bol'shaya Berdichevskaya Str., 10 008  Zhytomyr, UKRAINE \\
Phone: +38 -- (066) -- 959 50 34, \\
Email: esevostyanov2009@gmail.com

\begin{thebibliography}{99}
{\small

\bibitem[A]{A} T.~Adamowicz, ''Prime ends in metric spaces and boundary extensions of
mappings'', www. arxiv. org, arXiv:1608.02393.

\bibitem[ABBS]{ABBS} T.~Adamowicz, A.~Bj\"{o}rn, J.~Bj\"{o}rn,
N.~Shanmugalingam, ''Nageswari Prime ends for domains in metric
spaces'', {\it Adv. Math.}, V.~\textbf{238} (2013), 459-–505.

\bibitem[BB]{BB} A.~Bj\"{o}rn, J.~Bj\"{o}rn, ''Nonlinear Potential Theory on Metric Spaces,''
in: EMS Tracts in Mathematics, V.~\textbf{17}, European Math. Soc.,
Zurich, 2011.

\bibitem[Ch]{Ch} J.~Cheeger, ''Differentiability of Lipschitz functions on metric measure spaces,''
{\it Geom. Funct. Anal.}, V.~\textbf{9} (1999), 428-–517.

\bibitem[GRY]{GRY} V.~Ya.~Gutlyanskii, V.~I.~Ryazanov, E.~Yakubov,
''The Beltrami equations and prime ends'', {\it Ukr. Mat. Visn.},
V.~\textbf{12} (2015), no.~1,  27-–66; transl. {\it J. Math. Sci.
(N.Y.)}, V.~\textbf{210} (2015), no.~1, 22–-51.

\bibitem[Fu]{Fu} B.~Fuglede, ''Extremal length and functional
completion'', {\it Acta Math.}, V.~\textbf{98} (1957), 171--219.


\bibitem[IR]{IR} A.~Ignat'ev and V.~Ryazanov, ''Finite mean oscillation in mapping
theory'' (Russian), {\it Ukr. Mat. Visn.}, V.~\textbf{2} (2005),
no.~3, 395--417, translation in {\it Ukr. Math. Bull.},
V.~\textbf{2} (2005), no.~3, 403–-424.

\bibitem[KR]{KR}
D.A.~Kovtonyuk, V.I.~Ryazanov,  ''On the theory of prime ends for
space mappings'', {\it Ukr. Mat. Zh.}, V.~\textbf{67} (2015), no.~4,
467–-479; transl. {\it Ukrainian Math. J.}, V.~\textbf{67}(2015),
no.~4, 528–-541.

\bibitem[Ku]{Ku}
K.~Kuratowski, {\it Topology,} v.~2, New York--London: Academic
Press, 1968.

\bibitem[MRV]{MRV} O.~Martio, S.~Rickman,
J.~V\"{a}is\"{a}l\"{a}, ''Topological and metric properties of
quasiregular mappings'', {\it Ann. Acad. Sci. Fenn. Ser. A1.},
V.~\textbf{488} (1971),  1--31.

\bibitem[MRSY]{MRSY} O.~Martio, V.~Ryazanov, U.~Srebro and E.~Yakubov,
{\it Moduli in Modern Mapping Theory}, New York: Springer Science +
Business Media, LLC, 2009.

\bibitem[Na]{Na} R.~N\"{a}kki, ''Prime ends and quasiconformal
mappings'', {\it J. Anal. Math.}, V.~\textbf{35} (1979), 13--40.


\bibitem[NP]{NP} R.~N\"{a}kki and B.~Palka, ''Uniform equicontinuity
of  quasiconformal mappings'', {\it Proc. Amer. Math. Soc.},
V.~\textbf{37} (1973), no.~2, 427--433.

\bibitem[Ri]{Ri} S.~Rickman, {\it Quasiregular mappings},
Berlin etc.: Springer-Verlag, 1993.

\bibitem[RS]{RS} V.~Ryazanov, R.~Salimov,
''Weakly planar spaces and boundaries in the theory of mappings''
(Russian), {\it Ukr. Mat. Visn.}, V.~\textbf{4} (2007), no.~2,
199--234, translation in {\it Ukr. Math. Bull.}, V.~\textbf{4}
(2007), no.~2, 199–234.


\bibitem[Sm]{Sm} E.S.~Smolovaya, ''Boundary behavior of ring $Q$-homeomorphisms in
metric spaces'', {\it Ukrainian Math. J.}, V.~\textbf{62} (2010),
no.~5, 785–-793.

\bibitem[Sev$_1$]{Sev$_1$} E.~Sevost'yanov, ''On boundary behavior of mappings in terms of prime
ends'', www. arxiv. org,  arXiv:1602.00660.

\bibitem[Sev$_2$]{Sev$_2$} E.~Sevost'yanov, ''Local and boundary
behavior of maps in metric spaces'', {\it Algebra i analiz},
V.~\textbf{28} (2016), no.~6, 118--146 (in Russian); translation
in {\it St. Petersburg Math. J.}, V.~\textbf{28} (2017), no.~6,
P.~807--824.


\bibitem[Vu]{Vu} M.~Vuorinen, ''Exceptional sets and boundary behavior of quasiregular
mappings in $n$-space'', {\it Ann. Acad. Sci. Fenn. Ser. A 1. Math.
Dissertationes}, V.~\textbf{11}(1976), P.~1--44.}

\end{thebibliography}
\end{document}